\documentclass[11pt]{amsart}
\usepackage{amsfonts}
\usepackage{fullpage, amsmath, amsthm,amsfonts,amssymb,stmaryrd, mathrsfs}
\usepackage{hyperref}

\usepackage[pdftex]{color}
\usepackage[all]{xy}

\numberwithin{equation}{subsection}

\newtheorem{theorem}[subsection]{Theorem}
\newtheorem{lemma}[subsection]{Lemma}
\newtheorem{corollary}[subsection]{Corollary}

\newtheorem{proposition}[subsection]{Proposition}
\theoremstyle{definition}

\newtheorem{construction}[subsection]{Construction}

\newtheorem{remark}[subsection]{Remark}

\newtheorem{notation}[subsection]{Notation}

\def\calA{\mathcal{A}}

\def\calF{\mathcal{F}}
\def\calH{\mathcal{H}}
\def\calI{\mathcal{I}}
\def\calM{\mathcal{M}}
\def\calO{\mathcal{O}}
\def\calR{\mathcal{R}}
\def\calT{\mathcal{T}}

\def\gotha{\mathfrak{a}}
\def\gothc{\mathfrak{c}}
\def\gothd{\mathfrak{d}}
\def\gothn{\mathfrak{n}}

\def\gothC{\mathfrak{C}}

\def\gothp{\mathfrak{p}}
\def\gothq{\mathfrak{q}}

\def\CC{\mathbb{C}}
\def\FF{\mathbb{F}}
\def\GG{\mathbb{G}}

\def\PP{\mathbb{P}}
\def\QQ{\mathbb{Q}}
\def\RR{\mathbb{R}}
\def\TT{\mathbb{T}}
\def\ZZ{\mathbb{Z}}

\def\bfe{\mathbf{e}}

\def\bfk{\mathbf{k}}

\def\bfp{\mathbf{p}}
\def\bfq{\mathbf{q}}

\def\bfM{\mathbf{M}}
\def\bfN{\mathbf{N}}

\def\bfT{\mathbf{T}}

\def\scrF{\mathscr{F}}
\def\scrH{\mathscr{H}}
\def\ttD{\mathtt{D}}
\def\ttN{\mathtt{N}}
\def\ttS{\mathtt{S}}

\DeclareMathOperator{\End}{End}

\DeclareMathOperator{\Hom}{Hom}

\DeclareMathOperator{\Lie}{Lie}
\DeclareMathOperator{\rank}{rank}

\DeclareMathOperator{\Spec}{Spec}

\newcommand{\tor}{\mathrm{tor}}
\newcommand{\Gr}{\mathrm{Gr}}
\newcommand{\longto}{\longrightarrow}

\newcommand{\dR}{\mathrm{dR}}

\newcommand{\DP}{\mathrm{DP}}
\newcommand{\cris}{\mathrm{cris}}
\newcommand{\ex}{\mathrm{ex}}

\newcommand{\PR}{\mathrm{PR}}
\newcommand{\ps}{\mathrm{ps}}
\newcommand{\Ra}{\mathrm{Ra}}
\newcommand{\univ}{\mathrm{univ}}

\newcommand{\Sh}{\mathrm{Sh}}

\DeclareMathOperator{\GL}{GL}

\RequirePackage[normalem]{ulem} 
\RequirePackage{color}\definecolor{RED}{rgb}{1,0,0}\definecolor{BLUE}{rgb}{0,0,1} 

\begin{document}
\title[Partial Hasse invariants on splitting models]{Partial Hasse invariants on splitting models of Hilbert modular varieties}

\author{Davide A. Reduzzi}
\author{Liang Xiao}
\address{Liang Xiao, Department of Mathematics, University of Connecticut, Storrs,
Mathematical Sciences
Building 210, Unit 3009, Storrs, Connecticut 06269, USA}
\email{liang.xiao@uconn.edu}
\thanks{D.R. was partially supported by an AMS-Simons Research Travel Grant. L.X. is partially supported by Simons Collaboration Grant \#278433, NSF grant DMS-1502147,  and CORCL research grant from University of California, Irvine.}

\begin{abstract}
Let $F$ be a totally real field of degree $g$, and let $p$ be a prime number. We construct $g$ partial Hasse invariants on the characteristic $p$ fiber of the Pappas-Rapoport splitting model of the Hilbert modular variety for $F$ with level prime to $p$, extending the usual partial Hasse invariants defined over the Rapoport locus. In particular, when $p$ ramifies in $F$, we solve the problem of lack of partial Hasse invariants.
Using the stratification induced by these generalized partial Hasse invariants on the splitting model, we prove in complete generality the existence of Galois pseudo-representations attached to Hecke eigenclasses of paritious weight occurring in the coherent cohomology of Hilbert modular varieties $\mathrm{mod}$ $p^m$, extending a previous result of M. Emerton and the authors which required $p$ to be unramified in $F$.
\end{abstract}

\subjclass[2010]{11F41 (primary), 14G35 11G18 11F80 (secondary).}
\keywords{Partial Hasse invariants; Pappas-Rapoport splitting models; Hilbert modular forms; Galois representations}

\maketitle

\setcounter{tocdepth}{1}
\tableofcontents

\section{Introduction}
Let $F$ be a totally real field of degree $g>1$ over $\QQ$, and let $p$ be a prime number. Fix a large enough finite extension $\FF$ of $\FF_p$. The (characteristic $p$ fiber of the) \emph{Deligne-Pappas moduli space} $\calM^\DP_\FF$ parameterizes $g$-dimensional abelian schemes $A/S$ defined over an $\FF$-scheme $S$ and endowed with an action of the ring of integers of $F$, a polarization, and a suitable prime-to-$p$ level structure (cf. \cite{deligne-pappas}). The scheme $\calM^\DP_\FF$ is normal, and its smooth locus $\calM^\Ra_\FF$, called the \emph{Rapoport locus}, parameterizes those abelian schemes $A/S$ whose sheaf of invariant differentials is locally free of rank one as an $(\calO_F\otimes_\ZZ\calO_S)$-module (cf. \cite{rapoport}). 

In \cite{goren-hasse} and \cite{AG}, F. Andreatta and E. Goren construct some modular forms defined over the Rapoport locus, called the \emph{partial Hasse invariants}, which factor the determinant of the Hasse-Witt matrix of the universal abelian scheme $\calA^\Ra_\FF$ with respect to the action of the totally real field $F$.

When $p$ is unramified in $F$ there are exactly $g$ partial Hasse invariants, which give rise to a good stratification of the Hilbert moduli scheme $\calM^\Ra_\FF=\calM_\FF^\DP$ (cf. \cite{goren-oort} and \cite{AG}). On the other hand, when $p$ ramifies in $F$ the Rapoport locus is open and dense in $\calM^{\DP}_\FF$ with complement of codimension two, and the partial Hasse invariants of \cite{AG} do not extend to the Deligne-Pappas moduli space. In addition, the number of such operators is strictly less than $g$. For example, when $p$ is totally ramified in $F$, only one partial Hasse invariant is defined in \cite{AG} on $\calM^\Ra_\FF$: it is a $g$th root of the determinant of the Hasse-Witt matrix, up to sign. 
The lack of partial Hasse invariants when $p$ ramifies in $F$ was in particular an obstruction in extending to the ramified settings the results proved in the unramified case by M. Emerton and the authors in \cite{emerton-reduzzi-xiao}.

To remedy this, we work in this paper with the (characteristic $p$ fiber of the) \emph{splitting model} of the Hilbert modular scheme constructed by G. Pappas and M. Rapoport in \cite{pappas-rapoport}, and made explicit by S. Sasaki in \cite{sasaki}. This is a smooth scheme $\calM^\PR_\FF$ over $\FF$ endowed with a birational morphism $\calM^\PR_\FF\to\calM^\DP_\FF$ which is an isomorphism if and only if $p$ is unramified in $F$, and which induces an isomorphism from a suitable open dense subscheme of $\calM^\PR_\FF$ onto $\calM_\FF^\Ra$. There is a natural notion of automorphic line bundles on $\calM^\PR_\FF$, and hence of Hilbert modular forms.

We construct $g$ modular forms of non-parallel weight defined over the entire Pappas-Rapoport splitting model $\calM_\FF^\mathrm{PR}$, and extending the classical Hasse invariants over the Rapoport locus. Moreover, when $p$ ramifies in $F$ some of the operators we construct do not have a classical counterpart (see below).

We briefly discuss some of the ideas of this construction. Let us assume for simplicity that $p\calO_F = \gothp^e$ for some integer $e\geq 1$, and that the inertial degree of $p$ in $F$ is equal to $f$. Let $\varpi_\gothp$ denote a choice of uniformizer for the completed local ring $\calO_{F_\gothp}$, and denote by $\FF$ the residue field of $F$ at $\gothp$.
Let $\tau_1, \dots, \tau_f: \FF \to \overline \FF$ denote the embeddings of $\FF$ into its algebraic closure, ordered so that $\sigma \circ \tau_i = \tau_{i+1}$, where $i$ stands for $i\pmod f$, and $\sigma$ denotes the arithmetic Frobenius. For an abelian scheme $A/S$ defined over an $\FF$-scheme $S$ and endowed with real multiplication by $\calO_F$, we denote by $\omega_{A/S,j}$ the direct summand of the sheaf of invariant differentials of $A/S$ on which $W(\FF)\subset \calO_{F_\gothp}$ acts through
$\tau_j$. The sheaf $\omega_{A/S,j}$ is a locally free $\calO_S$-module of rank $e$. 

The Pappas-Rapoport splitting model $\calM_\FF^\PR$ parameterizes isomorphism classes of  tuples $(A, \lambda, i, \underline \scrF =(\scrF_j^{(l)})_{j,l})/S$ where $A$ is a Hilbert-Blumenthal abelian scheme defined over an $\FF$-scheme $S$, endowed with a polarization $\lambda$ and a prime-to-$p$ level structure $i$, and for each $j=1,\dots,f$ we are given a filtration of $\omega_{A/S,j}$:
\[
0 = \scrF_j^{(0)} \subsetneq \scrF_j^{(1)} \subsetneq \cdots \subsetneq \scrF_j^{(e)} = \omega_{A/S, j}
\]
by $\calO_F$-stable $\calO_S$-subbundles. We further require that each subquotient of the above filtrations is a locally free $\calO_S$-module of rank one, and that it is annihilated by the action of $[\varpi_\gothp]$ (cf. Subsection \ref{S:moduli HBAS}). We point out that in general the splitting model depends upon the choice of ordering of the $p$-adic embeddings of $F$. This dependence disappears when considering characteristic $p$ fibers, so we ignore the issue in this introduction, but cf. Remark \ref{R:dependence-on-embeddings}.

We observe that when $A$ satisfies the Rapoport condition, so that the sheaf of invariant differentials $\omega_{A/S}$ is an invertible $(\calO_S\otimes\calO_F)$-module, there is exactly one possible filtration on each of the sheaves $\omega_{A/S,j}$, namely the one obtained by considering increasing powers of the uniformizer. Moreover, when $e=1$, we have $\calM^\PR_\FF=\calM^\DP_\FF=\calM^\Ra_\FF$.

The subquotients $\omega_{j,l}:=\scrF_j^{(l),\mathrm{univ}}/\scrF_j^{(l-1),\mathrm{univ}}$ of the universal filtration over $\calM_\FF^\PR$ define automorphic line bundles over the splitting model (cf. Notation \ref{N:M Sh A}), so that we have a good notion of Hilbert modular forms as (roughly) elements of
\[
H^0\big(\calM^\PR_\FF,\otimes_{j,l}\omega_{j,l}^{\otimes k_{j,l}}).
\]
To define suitable generalizations of the partial Hasse invariants, it is natural to look at the action of the Verschiebung map on the invariant differentials of the universal abelian scheme over $\calM_\FF^\PR$, and it is not difficult to see that the Verschiebung map preserves the filtrations $\scrF_j^{(l)}$ and hence induces homomorphisms:
\[
V_j^{(l)} : \scrF_j^{(l)} / \scrF_j^{(l-1)} \to \big(\scrF_{j-1}^{(l)} / \scrF_{j-1}^{(l-1)} \big)^{(p)}.
\]
Unfortunately, one can check that when $e>1$ the zero locus of $V_j^{(l)}$ on $\calM^\PR_\FF$ is not irreducible. For example, when $e=2$ and $f=1$, the splitting model is obtained by blowing up the Deligne-Pappas moduli space in correspondence of its singularities (which are isolated points). The zero locus of the Verschiebung maps coincide with the union of the zero set of the (unique) classical partial Hasse invariant, together with the exceptional $\PP^1$'s attached via the blow-ups. 

In order to find a good notion of partial Hasse invariant on the splitting model, one needs to separate these irreducible components. To do so, in Section \ref{S:generalized Hasse} we construct two types of ``\emph{generalized partial Hasse invariants}":
\begin{align*}
\textrm{when }l >1,\quad &
m_j^{(l)}: \scrF_j^{(l)} / \scrF_j^{(l-1)} \to \scrF_j^{(l-1)} / \scrF_j^{(l-2)}, \textrm{ and}
\\
\textrm{when }l =1,\quad &
\mathrm{Hasse}_j: \scrF_j^{(1)} \to \big( \omega_{A/S,j-1} / \scrF_{j-1}^{(e-1)} \big)^{(p)},
\end{align*}
where the first morphism is essentially given by multiplication by $[\varpi_\gothp]$ (cf. Construction \ref{C:m}), and the second morphism is given by first ``dividing by $[\varpi^{e-1}_\gothp]$", and then applying the Verschiebung map (cf. Construction \ref{C:Hasse_}). As Hilbert modular forms, the partial Hasse invariant $m_j^{(l)}$ has weight $\omega_{j,l}^{\otimes-1}\otimes\omega_{j,l-1}$, while the partial Hasse invariant $\mathrm{Hasse}_j$ has weight
$\omega_{j,1}^{\otimes-1}\otimes\omega_{j-1,e}^{\otimes p}$.

One can then see (cf. Lemma \ref{L:factorization of V}) that the map $V_j^{(l)}$ factors as the composition:
\[
(m_{j-1}^{(l+1)})^{(p)}
\circ \cdots \circ (m_{j-1}^{(e)})^{(p)} \circ \mathrm{Hasse}_j \circ m_j^{(2)} \circ \cdots \circ m_j^{(l)},
\]
which explains the existence of the many irreducible components of the zero locus of $V_j^{(l)}$.

When restricted to the Rapoport locus of the splitting model, the operators $\mathrm{Hasse}_j$ coincide with the classical partial Hasse invariants constructed by Goren and Andreatta-Goren. On the other hand, the operators $m_j^{(l)}$ are invertible on the Rapoport locus. (For instance, when $e=2$ and $f=1$, the zero set of the operator $m_1^{(2)}$ is given by the exceptional $\PP^1$'s mentioned earlier).

Using crystalline deformation theory, we prove in Theorem \ref{T:smoothness of ramified GO strata} that the generalized partial Hasse invariants cut out proper smooth divisors with simple normal crossings on $\calM^\PR_\FF$. We point out that our geometric construction of these divisors is new, and we expect them to carry interesting arithmetic information.  For instance, it seems plausible to expect that they have a natural global description, and it would be interesting to study their cohomology, as it is done in the series of works by Y. Tian and the second named author \cite{tian-xiao1, tian-xiao2, tian-xiao3}, in which $p$ is assumed to be unramified. We also remark that a related but different-looking stratification is considered by S. Sasaki \cite{sasaki}.

As an application of our constructions, we can extend to the ramified settings the aforementioned results of Emerton and the authors (cf. \cite{emerton-reduzzi-xiao}) in which Galois representations were attached to torsion classes in the coherent cohomology of Hilbert modular variety, under the assumption that $p$ was unramified in $F$. (The interest in associating Galois representations to such torsion classes originated from the groundbreaking work of F. Calegari and D. Geraghty \cite{CG1, CG2}; their results 
were the main motivation behind our construction of the generalized partial Hasse invariants).
More precisely, denote by $\Sh^{\PR,\tor}$ a toroidal compactification of the splitting model for the Hilbert modular Shimura variety\footnote{We will explain in Section 2 the difference between this Shimura variety and the moduli space described earlier.  It is safe for the reader to ignore the difference here.} for $F$ with level $\Gamma_{00}({\mathcal{N}})$ (for some prime-to-$p$ ideal ${\mathcal{N}}$ of $\calO_F$ which is sufficiently divisible). Let $\ttD$ denote the boundary divisor of $\Sh^{\PR,\tor}$ and denote by $\omega^\kappa$ the automorphic line bundle of paritious weight $\kappa$ on $\Sh^{\PR,\tor}$ (cf. Section \ref{S:geometric HMF} for the definitions of these line bundles). Fix an integer $m \geq 1$ and set $R_m := \calO / (\varpi)^m$ where $\calO$ is the ring of integer in a sufficiently large finite extension of $\QQ_p$, and $\varpi$ is a uniformizer. Let $\ttS$ denote a finite set of places of $F$ containing all primes dividing $p{\mathcal{N}}$ and all archimedean places. Let $G_{F, \ttS}$ denote the Galois group of a maximal algebraic extension of $F$ unramified outside $\ttS$. We can prove the following (cf. Corollary~\ref{C:main}):
\begin{theorem}
For any Hecke eigenclass $c\in H^{\bullet}(\Sh^{\PR, \tor}_{R_m}
,\omega_{R_m}^{\kappa}(-\mathtt{D}))$, there is a continuous
$R_{m}$-linear, two-dimensional pseudo-representation $\tau_{c}$ of the Galois
group $G_{F,\ttS}$ such that
\[
\tau_{c}(\operatorname*{Frob}\nolimits_{\gothq})= a_{\gothq}%
\]
for all finite primes $\gothq$ of $F$ outside $\ttS$, where $a_{\gothq}$ is the eigenvalue for the Hecke operator $T_\gothq$ acting on $c$.
\end{theorem}

The theorem is proved essentially the same way as in \cite{emerton-reduzzi-xiao}, and using the following additional facts: (1) that there is ``sufficient supply" of partial Hasse invariants on the splitting model of the Hilbert modular Shimura variety, and their weights generate a positive open cone containing an ample automorphic line bundle; (2) that the divisors attached to the generalized partial Hasse invariants are proper and smooth with simple normal crossings.

The paper is organized as follows: in Section 2 we recall a few facts about Pappas-Rapoport splitting models for Hilbert modular schemes, and we prove the existence of a canonical Kodaira-Spencer filtration on their cotangent bundles. In Section 3 we define the ``generalized partial Hasse invariants", we prove that they cut out proper and smooth divisors with simple normal crossings, and we construct suitable canonical invertible sections $b_\tau$ over such divisors. In Section 4 we apply the results from Section 3 to prove (in the ramified settings) the above theorem on the existence of Galois representations attached to torsion Hilbert modular eigenclasses.

\subsection*{Acknowledgements}
We thank Matthew Emerton, Yichao Tian, and Xinwen Zhu for very helpful discussions during the preparation of this paper. We thank Georgios Pappas and Shu Sasaki for their interests in our results.

\section{Pappas-Rapoport splitting models for Hilbert modular schemes}

We begin by recalling the construction, due to Pappas and  Rapoport, of the splitting models for Hilbert modular schemes, following Pappas-Rapoport \cite{pappas-rapoport} and Sasaki \cite{sasaki}.  We assume that the reader is familiar with the construction of integral models over $\ZZ_p$ of Hilbert modular varieties attached to a totally real field $F$ in which the prime $p$ is unramified (see for example \cite{emerton-reduzzi-xiao}). We will try to make our presentation self-contained. After giving the basic definitions, we will focus on the new phenomena that occur when $p$ ramifies in $F$. 

\subsection{Setup}
Let $\overline \QQ$ denote the algebraic closure of $\QQ$ inside $\CC$.
We fix a rational prime $p$ and a field isomorphism $\overline \QQ_p \simeq \CC$.
Base changes of algebras and schemes will often be denoted by a subscript, if no confusion arises.

Let $F$ be a totally real field of degree $g>1$, with ring of integers $\calO_F$ and group of totally positive units $\calO_F^{\times, +}$.
Denote by $\gothd:=\gothd_F$ the different ideal of $F/ \QQ$.
Let $\gothC: = \{ \gothc_1, \dots, \gothc_{h^+}\}$ be a fixed set of representatives for the elements of the narrow class group of $F$, chosen to be coprime to $p$.

We fix a large enough coefficient field $E$ which is a finite Galois extension of $\QQ_p$ inside $\overline \QQ_p$. We require that $E$ contains the images of all $p$-adic embeddings of $F(\sqrt u; u \in \calO_F^{\times, +})$ into $\overline \QQ_p$.\footnote{The additional square roots of units are introduced for a technical reason; see Notation~\ref{N:M Sh A}.}
Let $\calO$ denote the valuation ring of $E$; choose a uniformizer $\varpi$ of $\calO$ and denote by $\FF$ the residue field.

We write the prime ideal factorization of $p \calO_F$ as $\gothp_1^{e_1} \cdots \gothp_r^{e_r}$, where $r$ and $e_i$ are positive integers.
Set $\FF_{\gothp_i} = \calO_F / \gothp_i$ and $f_i = [\FF_{\gothp_i}: \FF_p]$.
Let $\overline \FF_p$ denote the residue field of $\overline \ZZ_p$, and let $\sigma$ denote the arithmetic Frobenius on $\overline \FF_p$.
We label the embeddings of $\FF_{\gothp_i}$ into $\overline \FF_p$ (or, equivalently, into $\FF$) as $\{\tau_{\gothp_i, 1}, \dots, \tau_{\gothp_i, f_i}\}$ so that $\sigma \circ \tau_{\gothp_i, j} = \tau_{\gothp_i, j+1}$ for all $j$, with the convention that $\tau_{\gothp_i, j}$ stands for $\tau_{\gothp_i, j \pmod {f_i}}$.
For each $\gothp_i$, denote by $F_{\gothp_i}$ the completion of $F$ for the $\gothp_i$-adic topology. Let  $W(\FF_{\gothp_i})$ denote the ring of integers of the maximal subfield of ${F_{\gothp_i}}$ unramified over $\QQ_p$. The residue field of $W(\FF_{\gothp_i})$ is identified with $\FF_{\gothp_i}$. Each embedding $\tau_{\gothp_i, j}:\FF_{\gothp_i} \to \FF$ of residue fields induces an embedding $W(\FF_{\gothp_i}) \to \calO$ which we denote by the same symbol.

Let $\Sigma$ denote the set of embeddings of $F$ into $\overline \QQ$, which is further identified with the set of embeddings of $F$ into $\CC$ or $\overline \QQ_p$ or $E$.
Let $\Sigma_{\gothp_i}$ denote the subset of $\Sigma$ consisting of all the $p$-adic embeddings of $F$ inducing the $p$-adic place $\gothp_i$.
For each $i$ and each $j =1, \dots, f_i$, there are exactly $e_i$ elements in $\Sigma_{\gothp_i}$ that induce the embedding $\tau_{\gothp_i, j}:W(\FF_{\gothp_i}) \to \calO$; we label these elements as $\tau_{\gothp_i, j}^1, \dots, \tau_{\gothp_i, j}^{e_i}$.  \emph{There is no canonical choice of such labeling, but we fix one for the rest of this paper}.

We choose a uniformizer $\varpi_i$ for the ring of integers $\calO_{F_{\gothp_i}}$ of $F_{\gothp_i}$.  Let $E_{\gothp_i}(x)$ denote the minimal polynomial of $\varpi_i$ over the ring $W(\FF_{\gothp_i})$: it is an Eisenstein polynomial.
Using the embedding $\tau_{\gothp_i,j}$, we can view this polynomial as an element $E_{\gothp_i,j}(x) := \tau_{\gothp_i,j}(E_{\gothp_i}(x))$  of $\calO[x]$.  We have:
\[
E_{\gothp_i,j}(x) = (x - \tau_{\gothp_i, j}^1(\varpi_i)) \cdots (x - \tau_{\gothp_i, j}^{e_i}(\varpi_i))
\quad.
\]

\subsection{Pappas-Rapoport splitting models}
\label{S:moduli HBAS}
Let $S$ be a locally Noetherian $\calO$-scheme. A \emph{Hilbert-Blumenthal abelian $S$-scheme} (HBAS) with real multiplication by $\calO_F$ is the datum of an abelian $S$-scheme $A$ of relative dimension $g$, together with a ring embedding $ \calO_F \to \End_SA$.
We have natural direct sum decompositions
\[
\omega_{A/S} = \bigoplus_{i=1}^r \omega_{A/S, \gothp_i} =  \bigoplus_{i=1}^r \bigoplus_{j = 1}^{f_i}\omega_{A/S, \gothp_i, j},
\]
\[
\calH_\dR^1(A/S) = \bigoplus_{i =1}^r \bigoplus_{j = 1}^{f_i}
\calH_\dR^1(A/S)_{\gothp_i, j},
\]
where $W(\FF_{\gothp_i}) \subseteq \calO_{F_{\gothp_i}}$ acts on $\omega_{A/S, \gothp_i,j}$ (resp. on $\calH_\dR^1(A/S)_{\gothp_i, j}$) via $\tau_{\gothp_i, j}$. Moreover, since
$\calH^1_\dR(A / S)$ is a locally free $\calO_F \otimes_\ZZ \calO_S$-module of rank two (cf. \cite[Lemme 1.3]{rapoport}), 
each $\calH_\dR^1(A/S)_{\gothp_i, j}$ is a locally free sheaf over $S$ of rank $2e_i$.

Let $\gothc \in \gothC$ be a fractional ideal of $F$ (coprime to $p$), with cone of positive elements $\gothc^+$.
We say a HBAS $A$ over $S$ is \emph{$\gothc$-polarized} if
there is an $S$-isomorphism $\lambda: A^\vee \to A \otimes _{\calO_F} \gothc$ of HBAS's under which the symmetric elements (resp. the polarizations) of $\Hom_{\calO_F}(A, A^\vee)$ correspond to the elements of $\gothc$ (resp. $\gothc^+$) in $\Hom_{\calO_F}(A, A \otimes_{\calO_F} \gothc)$.

We remark that for a $\gothc$-polarized HBAS $(A,\lambda)/S$, each $\omega_{A/S, \gothp_i, j}$ is a \emph{locally free sheaf over $S$ of rank $e_i$}.
This fact is proved, for example, in \cite[Proposition~2.11]{vollaard}.  Or we can see this directly by looking at a closed point on each connected component of $S$ and hence reduce to the case when $S =\Spec k$, with $k$ an algebraically closed field.
The polarization $\lambda$ induces a non-degenerate symplectic pairing on each $H^1_\dR(A/S)_{\gothp_i, j}$, and $\omega_{A/S,\gothp_i,j}$ is an isotropic subspace for such pairing. The dimension of $\omega_{A/S}$ is half of the dimension of $H^1_\dR(A/S)$.  So the dimension of each $\omega_{A/S,\gothp_i,j}$ is forced to be $e_i$.

Let $\mathcal{N}$ be a non-zero proper ideal of $\calO_F$ coprime to $p$. A \emph{$\Gamma_{00}({\mathcal{N}})$-level structure} on a HBAS $A$ over $S$ is an $\calO_F$-linear closed embedding of $S$-schemes
$ i: \mu_{\mathcal{N}} \to A$, where $\mu_{\mathcal{N}}:=(\gothd^{-1} \otimes_\ZZ \GG_m)[{\mathcal{N}}]$ is the Cartier dual of the constant $S$-group scheme $\calO_F/\mathcal{N}$.

Fix a non-zero proper ideal ${\mathcal{N}}$ of $\calO_F$ coprime to $p$, and assume that $\mathcal{N}$ does not divide neither $2\calO_F$ nor $3\calO_F$. Denote by $\underline \calM^\PR_\gothc = \underline \calM^\PR_{\gothc, {\mathcal{N}}}$ the functor that assigns to a locally Noetherian $\calO$-scheme $S$ the set of isomorphism classes of tuples $(A, \lambda, i, \underline \scrF)$, where:
\begin{itemize}
\item
$(A, \lambda)$ is a $\gothc$-polarized HBAS over $S$ with real multiplication by $\calO_F$,
\item
$i$ is a $\Gamma_{00}({\mathcal{N}})$-level structure,
\item $\underline \scrF$ is a collection $(\scrF_{\gothp_i, j}^{(l)})_{i = 1, \dots, r; j = 1, \dots, f_i;l = 0, \dots, e_i}$ of locally free sheaves over $S$ such that
\begin{itemize}
\item $0 =\scrF_{\gothp_i, j}^{(0)} \subsetneq \scrF_{\gothp_i, j}^{(1)} \subsetneq \cdots \subsetneq
\scrF_{\gothp_i, j}^{(e_i)} = \omega_{A/S, \gothp_i, j}$ and each $\scrF_{\gothp_{i, j}}^{(l)}$ is stable under the $\calO_F$-action (not just the action of $W(\FF_{\gothp_i})$),
\item each subquotient $\scrF_{\gothp_i, j}^{(l)} / \scrF_{\gothp_i, j}^{(l-1)}$ is a locally free $\calO_S$-module of rank one (and hence $\rank_{\calO_S} \scrF_{\gothp_i, j}^{(l)} = l$), and
\item the action of $\calO_F$ on each subquotient $\scrF_{\gothp_i, j}^{(l)} / \scrF_{\gothp_i, j}^{(l-1)}$ factors through $\calO_F \xrightarrow{\tau_{\gothp_i, j}^l} \calO$, or equivalently, this subquotient is annihilated by $[\varpi_i] - \tau_{\gothp_i,j}^l(\varpi_i)$, where $[\varpi_i]$ denotes the action of $\varpi_i$ as an element of $\calO_{F_{\gothp_i}}$.
\end{itemize}
\end{itemize}
We use $\underline \calM_\gothc^\DP$ to denote the functor obtained from $\underline \calM_\gothc^\PR$ by forgetting the filtrations $\underline \scrF$.

Both $\underline \calM_\gothc^\PR$ and $\underline \calM_\gothc^\DP$ carry an action of $\calO_F^{\times, +}$:
\begin{equation}
\label{E:action <u>}
\textrm{for }u \in \calO_F^{\times,+}, \quad \langle u \rangle: (A, \lambda, i, \underline \scrF) \longmapsto (A, u\lambda, i, \underline \scrF).
\end{equation}
The subgroup $(\calO_{F, {\mathcal{N}}}^\times)^2$ of $\calO_F^{\times,+}$ acts trivially on both spaces, where $\calO_{F, {\mathcal{N}}}^\times$ denotes the group of units that are congruent to $1$ modulo ${\mathcal{N}}$.

\begin{remark}
\label{R:dependence-on-embeddings}
\begin{enumerate}
\item
The definition of the functor $\underline \calM_\gothc^\PR$ on the category of locally Noetherian $\calO$-schemes depends upon the fixed choice of an ordering for the set of embeddings $\tau_{\gothp_i,j}^l$.
Although it appears that the base change of the functor $\underline \calM_\gothc^\PR$ to $\FF$ does not depend on such choice anymore, its relation to Hilbert modular forms is given through its integral model, and hence automorphic properties still depend on the ordering.

Related to this issue, the corresponding moduli space $\calM_\gothc^\DP$ may be descended to $\ZZ_p$, however, the moduli spacce $\calM_{\gothc}^\PR$ needs to be defined over the composite of all $\calO_{F_{\gothp}}$'s.

\item
The idea of Pappas and Rapoport (\cite{pappas-rapoport}) of introducing the filtration $\underline \scrF$ to define $\underline \calM_\gothc^\PR$, and the consequent existence of the forgetful morphism $\pi:\underline \calM_\gothc^\PR \to \underline \calM_\gothc^\DP$, is modeled on the following resolution of the affine Grassmannian for $\GL_2$
\[
\Gr_1 \tilde \times \Gr_1 \tilde \times \cdots \tilde \times \Gr_1 \longto \Gr_{\leq e},
\]
where the left hand side is the twisted product of $e$ copies of Grassmannians (projective lines in this case), and the right hand side is an affine Grassmannian variety.
\end{enumerate}

\end{remark}

We now discuss some properties of the moduli functors just introduced.  The following result is well known (cf. Pappas-Rapoport \cite{pappas-rapoport} and Sasaki \cite{sasaki}).

\begin{proposition}
\label{P:basic property of integral models}
\begin{itemize}
\item[(1)]
The functor $\underline \calM_\gothc^\mathrm{DP}$ (resp. $\underline \calM_\gothc^\mathrm{PR}$) is represented by an $\calO$-scheme of finite type that we denote $\calM_\gothc^\DP$ (resp. $\calM_\gothc^\PR$).
\item[(2)]
The moduli space  $\calM_\gothc^\mathrm{DP}$ is normal.
Let $\calM_\gothc^\mathrm{Ra}$ denote its smooth locus, called the \emph{Rapoport locus}.  Then $\calM_\gothc^\mathrm{Ra}
$ is the subscheme of $\calM_\gothc^\mathrm{DP}$ parameterizing those HBAS for which the cotangent space at the origin $\omega_{A/S}$ is a locally free $(\calO_{F} \otimes_{\ZZ} \calO_S)$-module of rank one.
\item
[(3)]
The natural morphism $ \pi:\calM_\gothc^\mathrm{PR} \to \calM_\gothc^\mathrm{DP}$ is projective, and it induces an isomorphism of a subscheme of $ \calM_\gothc^\mathrm{PR}$ onto $\calM_\gothc^\mathrm{Ra}$.
\item
[(4)]
If ${\mathcal{N}}$ is sufficiently divisible, the actions of $\calO_F^{\times, +} / (\calO_{F,{\mathcal{N}}}^\times)^2$ on $\calM_\gothc^\PR$ and $\calM_\gothc^\DP$ are free on geometric points. In particular, the corresponding quotients $ \Sh_\gothc^\PR$ and $ \Sh_{\gothc}^\DP$ are $\calO$-schemes of finite type, and the quotient morphisms are \'etale.
\end{itemize}

\end{proposition}
\begin{proof}
The representability of $\underline \calM_\gothc^\mathrm{DP}$ follows from Deligne-Pappas \cite{deligne-pappas}. The representability of $\underline \calM_\gothc^\mathrm{PR}$ and the projectivity of $\pi$ follow from the fact that $\underline \calM_\gothc^\mathrm{PR}$ is relatively representable over $\underline \calM_\gothc^\mathrm{DP}$ by a closed subscheme of a Grassmannian.\footnote{This uses the fact that for any $S$-point $(A, \lambda, i)$ of $\calM_\gothc^\DP$, each sheaf of differentials $\omega_{A/S, \gothp_i,j}$ is locally free of rank $e_i$.}
The second statement in the proposition is proved in \cite{deligne-pappas}.  The third statement follows since over $\calM_\gothc^\mathrm{Ra}$ we have a unique choice for the filtration $\underline \scrF$. The last statement can be proved in exactly the same way as in \cite[2.1.1]{emerton-reduzzi-xiao}.
\end{proof}

\begin{lemma}
\label{L:wedge of HdR}
Let $(A, \lambda)$ be a $\gothc$-polarized HBAS defined over a locally Noetherian scheme $S$. There is a canonical isomorphism:
\[
\wedge^2_{\calO_F \otimes_\ZZ \calO_S}\calH^1_\dR(A / S) \cong \gothc \gothd^{-1}\otimes_{\ZZ}\calO_S.
\]
induced by the polarization $\lambda$.
\end{lemma}
\begin{proof}
Recall that $\calH^1_\dR(A / S)$ is a locally free $\calO_F \otimes_\ZZ \calO_S$-module of rank two (cf. \cite[Lemme 1.3]{rapoport}).
The polarization induces an isomorphism:
\begin{align*}
 \gothc^{-1}\otimes_{\calO_F}\calH^1_\dR(A/S)  & \xrightarrow{\cong} \calH^1_\dR(A^\vee /S).
\end{align*}
The lemma follows by composing this with the identifications:
\begin{align*}
\calH^1_\dR(A^\vee /S) & \cong \calH om_{\calO_S}(\calH^1_{\dR}(A/S), \calO_S)
\\&
\cong \gothd^{-1} \otimes_{\calO_F}\calH om_{\calO_F \otimes_{\ZZ} \calO_S}(\calH^1_{\dR}(A/S),  \calO_F \otimes_{\ZZ} \calO_S).\qedhere
\end{align*}
\end{proof}

\begin{notation}
\label{N:M Sh A}
For $? \in\{\mathrm{DP}, \mathrm{PR}, \mathrm{Ra}\}$ we denote by $\calA_\gothc^?$ the universal abelian scheme over $\calM_\gothc^?$.  We set
\[
\calM^? := \coprod_{\gothc \in \gothC} \calM_\gothc^?, \quad \Sh^? := \coprod_{\gothc \in \gothC} \Sh_\gothc^?, \quad \textrm{and} \quad \calA^? := \coprod_{\gothc \in \gothC} \calA_\gothc^?.
\]
Notice that the universal abelian scheme $\calA^?$ may not descent to $\Sh^?$, so that $\Sh^?$ is a \emph{coarse} moduli space for HBAS endowed with a polarization class and a $\Gamma_{00}({\mathcal{N}})$-level structure.

Denote by $\omega_{\calA^?/\calM^?}$ the pull-back via the zero section of the sheaf of relative differentials of $\calA^?$ over $\calM^?$. We let $\underline \calF = (\calF_{\gothp_i,j}^{(l)})$ denote the universal filtration of $\omega_{\calA^\PR/\calM^\PR}$. For each $p$-adic embedding $\tau = \tau_{\gothp_i, j}^l$ of $F$ into $\overline \QQ_p$,
we set
\[
\dot\omega_\tau: = \calF_{\gothp_i, j}^{(l)} / \calF_{\gothp_i, j}^{(l-1)};
\footnote{The additional dot in the notation $\dot \omega_\tau$ is placed in order to distinguish this sheaf from its descent $\omega_\tau$ to $\Sh^\PR$, which will be introduced later.}
\]
it is an automorphic line bundle on the splitting model $\calM^\PR$.  While each individual $\dot \omega_\tau$ does not descend to $\calM^\DP$ in general, their tensor product $\otimes_{\tau \in \Sigma} \dot \omega_\tau$ does, as it is the Hodge  bundle $\wedge^\mathrm{top} \omega_{\calA/\calM}$.  We provide $\dot \omega_\tau$ with an action of $\calO_F^{\times, +}$ following \cite{DT}: a positive unit $u \in \calO_F^{\times, +}$ maps a local section $s$ of $\dot \omega_\tau$ to $u^{-1/2} \cdot \langle u \rangle^*(s)$, where $\langle u \rangle$ is defined by \eqref{E:action <u>}.  It is clear that this action factors through $\calO_F^{\times, +} / (\calO_{F,{\mathcal{N}}}^\times)^2$.

By Proposition~\ref{P:basic property of integral models}(2)(3), the sheaf $\omega_{\calA^\Ra / \calM^\Ra}$ is locally free of rank one over $\calO_{F} \otimes_{\ZZ} \calO_{\calM^\Ra}$, and we have for any $\tau \in \Sigma$:
\begin{equation}
\label{E:omega nice over Ra}
\dot \omega_{\tau}|_{\calM^\Ra} = \omega_{\calA^\Ra / \calM^\Ra} \otimes_{\calO_{F} \otimes_{\ZZ} \calO_{\calM^\Ra},  \tau \otimes 1} \calO_{\calM^\Ra}.
\end{equation}
The expression on the right hand side may not give a line bundle outside $\calM^\Ra$.

Similarly, for each $p$-adic embedding $\tau$ of $F$, we define

\begin{align*}
\dot \varepsilon_\tau & := \big( \wedge^2_{\calO_F\otimes_\ZZ \calO_{\calM^\PR} } \calH^1_\dR(\calA^\PR/{\calM^\PR}) \big) \otimes_{\calO_F\otimes_\ZZ \calO_{\calM^\PR}, \tau\otimes1} \calO_{\calM^\PR}
\\&
\cong (\gothc \gothd^{-1}\otimes_{\ZZ} \calO_{\calM^\PR} ) \otimes_{\calO_F\otimes_\ZZ \calO_{\calM^\PR}, \tau \otimes 1} \calO_{\calM^\PR}.
\end{align*}
While $\dot\varepsilon_\tau$ is a trivial line bundle, it carries a non-trivial action of $\calO_F^{\times,+} / (\calO_{F, {\mathcal{N}}}^\times)^2$ given as follows: a positive unit $u \in \calO_F^{\times, +}$ maps a local section $s$ of $\dot \varepsilon_\tau$ to $u^{-1} \cdot \langle u \rangle^*(s)$.
\end{notation}
Recall that we denote base change of schemes by a subscript.
\begin{lemma}
\label{L:relative ample}
The line bundle
\[
\dot \omega^\mathrm{ex}_\FF : = \bigotimes_{i = 1}^r \bigotimes_{j =1}^{f_i} \bigotimes_{l = 1}^{e_i} \dot \omega_{\tau_{\gothp_i, j}^l, \FF}^{\otimes2(l-e_i-1)} \otimes \dot \epsilon_{\tau_{\gothp_i, j}^l, \FF}^{\otimes(e_i+1-l)}
\]
defined over the special fiber $\calM^\mathrm{PR}_\FF$ of $\calM^\mathrm{PR}$
is relatively ample with respect to $\pi: \calM^\PR_\FF\to\calM^\mathrm{DP}_\FF$.
\end{lemma}
\begin{proof}
We thank Xinwen Zhu for his help on this proof.
It suffices to check the ampleness at each point $z$ of $\calM_\FF^\DP$ (say with residue field $k$). We denote by $A_z$ the universal abelian variety at the point $z$.
The fiber $\pi^{-1}(z)$ is the space parameterizing all possible filtrations $\scrF_{\gothp_i,j}^{(l)}$ on $\omega_{A_z/k, \gothp_i,j}$; so it is a closed subscheme of the space parameterizing submodules $ \scrF_{\gothp_i,j}^{(l)}$ of $\calH^1_\dR(A_z/k)_{\gothp_i,j} \cong( k[x]/(x^{e_i}) \big)^{\oplus 2}$ of rank $l$.
In particular $\pi^{-1}(z)$ is a closed subscheme of the product of affine Grassmannians for $\GL_2$ given by  $X:=\prod_{i=1}^r \prod_{j=1}^{f_i} \big( \Gr_{\leq 1} \times \cdots \times \Gr_{\leq e_i})$.
Note that
\[
\big( \dot \omega_{\tau_{\gothp_i,j}^{1}, \FF} \otimes \cdots \otimes \dot \omega_{\tau_{\gothp_i,j}^{l}, \FF} \big)^{\otimes-1}
\]
is an ample sheaf on the Grassmannian $\Gr_{\leq l}$ appearing as a factor of $X$ for $\gothp_i$ and $j$, and so is its square.
Taking the product of these squares gives the ample line bundle $\dot \omega_\FF^\ex$ over the ambient space $X$.  It follows that $\dot \omega_\FF^\ex$ is relative ample with respect to $\pi: \calM_\FF^\PR \to \calM_\FF^\DP$.
\end{proof}

\subsection{Kodaira-Spencer filtration}
Following the arguments of \cite[Proposition~14]{sasaki} one sees that the scheme $\calM^\mathrm{PR}$ is smooth over $\calO$ (cf. also \cite{pappas-rapoport}). For later use, we also need to construct a canonical Kodaira-Spencer-type filtration on the sheaf of relative differentials of the splitting model:

\begin{theorem}
\label{T:Kodaira-Spencer ramified}
The scheme $\calM^\mathrm{PR}$ is smooth over $\calO$.  Moreover, the sheaf of relative differentials $\Omega^1_{\calM^\mathrm{PR} / \calO}$ admits a canonical filtration whose successive subquotients are given by
\[
\dot  \omega_{\tau}^{\otimes 2}  \otimes_{\calO_{\calM^{\PR}}} \dot \varepsilon_{\tau}^{\otimes-1}\quad  \textrm{ for }\tau \in \Sigma. \footnotemark
\]
\end{theorem}\footnotetext{The relative differential sheaf is \emph{not} in general the direct sum of these sheaves, at least not as an $\calO_F \otimes_{\ZZ} \calO_{\calM^\PR}$-module.}
\begin{proof}
The first statement is proved in \cite{sasaki}, following \cite{pappas-rapoport}. We give here a combined proof of both statements in the theorem. Let $S_0 \hookrightarrow S$ be a closed immersion of locally Noetherian $\calO$-schemes whose ideal of definition $\calI$ satisfies $\calI^2 = 0$.  Consider an $S_0$-valued point $x_0 = (A_0, \lambda_0, i_0, \underline \scrF)$ of $\calM^\mathrm{PR}$.
To prove the smoothness of $\calM^\mathrm{PR}$, it suffices to show that there exists $x \in \calM^\mathrm{PR}(S)$ lifting $x_0$.

We first introduce some notation.
Let $\calH_\cris^1(A_0 / S_0)$ denote the crystalline cohomology sheaf of $A_0$ over $S_0$, which is locally free of rank two over $\calO_{S_0}^\cris \otimes_\ZZ \calO_F$ by \cite[Lemme 1.3]{rapoport}. The action of $\calO_F$ on $A_0$ induces a natural direct sum decomposition:
\[
\calH_\cris^1(A_0/S_0) = \bigoplus_{i =1}^r \bigoplus_{j = 1}^{f_i}
\calH_\cris^1(A_0/S_0)_{\gothp_i, j},
\]
so that $W(\FF_{\gothp_i}) \subseteq \calO_{F_{\gothp_i}}$ acts on $\calH_\cris^1(A_0/S_0)_{\gothp_i, j}$ via $\tau_{\gothp_i,j}$.
Moreover $\calH_\cris^1(A_0/S_0)_{\gothp_i, j}$ is a locally free module of rank two over
\[
\calO_{F_{\gothp_i}} \otimes_{W( \FF_{\gothp_i}), \tau_{\gothp_i,j}} \calO^\mathrm{cris}_{S_0} \cong \calO^\mathrm{cris}_{S_0} [x] / (E_{\gothp_i,j}(x)).
\]
Since $S$ is a PD-thickening of $S_0$, we can evaluate the crystalline cohomology over $S$ to obtain $\calH_\cris^1(A_0/S_0)_S$ and its direct summands $\calH_\cris^1(A_0/S_0)_{S,\gothp_i, j}$.

Similarly, the natural exact sequence
\[
0 \to \omega_{A_0/S_0} \to \calH^1_\cris(A_0/S_0)_{S_0} \to \Lie_{A_0^\vee/S_0} \to 0
\]
decomposes into the direct sum of exact sequences
\[
0 \to \omega_{A_0/S_0, \gothp_i,j} \to \calH^1_\cris(A_0/S_0)_{S_0, \gothp_i,j} \to \Lie_{A_0^\vee/S_0, \gothp_i,j} \to 0.
\]
Here each $\omega_{A_0/S_0, \gothp_i,j}$ is a (not necessarily locally free) coherent $\calO_{F_{\gothp_i}} \otimes_{W( \FF_{\gothp_i}), \tau_{\gothp_i,j}} \calO_{S_0}$-module, which is locally free and of rank $e_i$ as an $\calO_{S_0}$-module.  It is  locally a direct summand of $\calH^1_\cris(A_0/S_0)_{S_0, \gothp_i,j}$.

The polarization $\lambda_0: A_0^\vee \to A_0 \otimes_{\calO_F} \gothc$ induces a non-degenerate, symplectic pairing\footnote{To see this, we note that the polarization induces an isomorphism $\lambda_0^*:  \calH_\cris^1(A_0/S_0) \otimes_{\calO_F}\gothc \to \calH_\cris^1(A_0^\vee/S_0) \cong (\calH_\cris^1(A_0/S_0))^\vee$. Taking the $\tau_{\gothp_i,j}$-component gives the pairing.} 
\[
\langle\cdot, \cdot \rangle\colon
\calH^1_\cris(A_0/S_0)_{\gothp_i,j} \times \calH^1_\cris(A_0/S_0)_{\gothp_i,j} \to \calO_{S_0}^\cris, \ \textrm{such that:}
\]
\begin{equation}
\label{E:OF hermitian}
\langle ax , y \rangle  = \langle x,ay \rangle \textrm{ for }a \in \calO_{F_{\gothp_i}} \textrm{, }x,y \in \calH^1_\cris(A_0/S_0)_{\gothp_i,j},\mbox{ and}
\end{equation}
\begin{equation}
\label{E:strongly symplectic}
\langle ax, x \rangle = 0,\textrm{ for }a \in \calO_{F_{\gothp_i}} \textrm{, } x \in \calH^1_\mathrm{cris}(A_0/S_0)_{ \gothp_i,j}.
\end{equation}

Here the validity of condition~\eqref{E:strongly symplectic} can be verified as follows.
When the integer $2$ is not a zero divisor on $S$, this is clear by \eqref{E:OF hermitian} and the symplectic nature of the pairing.
In the general case, we may assume that $S_0 = \Spec R_0$ is affine and $x \in \calH^1_\mathrm{cris}(A_0/S_0)_{S', \gothp_i,j}$ is a section over some PD-thickening $S_0 \hookrightarrow S' = \Spec R'$ where $R'$ is Noetherian. Write $R'$ as a quotient of $\ZZ_p$-flat Noetherian algebra $\tilde R$, and let $\tilde R_{\mathrm{PD}}$ denote the PD-envelop of the quotient map $\tilde R \twoheadrightarrow R' \twoheadrightarrow R_0$, so that we may evaluate the crystalline cohomology on the PD-thickening $\tilde R_{\mathrm{PD}} \to R_0$.
Now $\tilde R_{\mathrm{PD}}$ is $\ZZ_p$-flat and so \eqref{E:strongly symplectic} holds over $\tilde R_{\mathrm{PD}}$; so it holds over $R$ as $R$ is a quotient of $\tilde R_\mathrm{PD}$ by the universal property of the PD-envelop.

The submodule $\omega_{A_0/S_0,\gothp_i,j}$ of $\calH^1_\cris(A_0/S_0)_{S_0,\gothp_i,j}$ is (maximal) isotropic with respect to the pairing $\langle\cdot, \cdot \rangle$.  In particular, each $\scrF_{\gothp_i,j}^{(l)}$ is contained in its annihilator with respect to $\langle\cdot,\cdot\rangle$.

By crystalline deformation theory (cf. \cite[pp. 116--118]{grothendieck},   \cite[Chap. II \S 1]{mazur-messing}), to lift the abelian variety $A_0$ together with the $\calO_F$-action, it suffices to lift each $\omega_{A_0/S_0, \gothp_i,j}$ to an $\calO_F$-stable submodule $\tilde \omega_{\gothp_i,j}$ of $\calH^1_\cris(A_0/S_0)_{S, \gothp_i,j}$ which is a locally free $\calO_S$-subbundle of rank $e_i$.
Once such lift is determined, the level structure $i_0$ lifts uniquely; the polarization $\lambda_0$ lifts if $\tilde \omega_{\gothp_i,j}$ is isotropic for the pairing $\langle \cdot, \cdot \rangle$.\footnote{Alternatively, one may quote the main result of \cite{vollaard} to show that the polarization lifts as long as we can lift the filtrations $\scrF_{\gothp_i,j}^{(l)}$, which might save a few lines. Unfortunately, there is a very minor gap in the proof of \cite[Proposition~2.11]{vollaard}, as the condition \eqref{E:strongly symplectic} was neglected.  For completeness, we give here a self-contained proof.}

To lift the $S_0$-point $x_0 = (A_0, \lambda_0, i_0, \underline \scrF)$ to an $S$-point, one needs to lift, for each $\tau_{\gothp_i,j}$, the entire filtration
\[
0 =\scrF_{\gothp_i, j}^{(0)} \subsetneq \scrF_{\gothp_i, j}^{(1)} \subsetneq \cdots \subsetneq
\scrF_{\gothp_i, j}^{(e_i)} = \omega_{A_0/S_0, \gothp_i, j} \subset \calH^1_\cris(A_0 /S_0)_{S_0,\gothp_i,j}
\]
to an $\calO_F$-stable filtration
\[
0 =\tilde \scrF_{\gothp_i, j}^{(0)} \subsetneq \tilde \scrF_{\gothp_i, j}^{(1)} \subsetneq \cdots \subsetneq
\tilde \scrF_{\gothp_i, j}^{(e_i)} = \tilde \omega_{\gothp_i, j} \subset \calH^1_\cris(A_0 /S_0)_{S, \gothp_i,j}
\]
such that:
\begin{itemize}
\item
each subquotient $\tilde \scrF_{\gothp_i, j}^{(l)} / \tilde \scrF_{\gothp_i, j}^{(l-1)}$ is a locally free sheaf of rank one over $S$,
\item
each $\tilde \scrF_{\gothp_i,j}^{(l)}$ is contained in its annihilator $(\tilde \scrF_{\gothp_i,j}^{(l)})^\perp$ with respect to the pairing $\langle \cdot , \cdot\rangle$, and
\item the action of $\calO_F$ on each subquotient $\tilde \scrF_{\gothp_i, j}^{(l)} / \tilde  \scrF_{\gothp_i, j}^{(l-1)}$ factors through $\calO_F \xrightarrow{\tau_{\gothp_i, j}^l} \calO$.
\end{itemize}
The smoothness of the moduli space, as well as the existence of the Kodaira-Spencer filtration, follow from understanding the lifts of each step of the filtration inductively, beginning with $\tilde \scrF_{\gothp_i,j}^{(1)}$ and climbing up to $\tilde \scrF_{\gothp_i, j}^{(e_i)}$.

Let $\scrH_{\gothp_i,j}^{(1)}$ denote the kernel of the map $[\varpi_i] - \tau_{\gothp_i, j}^{1}(\varpi_i)$ acting on the locally free rank two $\calO_S[x]/(E_{\gothp_i,j}(x))$-module $\calH^1_\mathrm{cris}(A_0/S_0)_{S, \gothp_i,j}$. (Recall that $[\varpi_i]$ denotes the action of $\varpi_i$ as an element of $\calO_{F_{\gothp_i}}$).
By condition \eqref{E:strongly symplectic}, the isotropic condition on a lift $\tilde \scrF_{\gothp_i,j}^{(1)}$ of $\scrF_{\gothp_i,j}^{(1)}$ is automatically satisfied. Observe that $\scrH_{\gothp_i,j}^{(1)}$ is a rank-two $\calO_S$-subbundle of $\calH^1_\mathrm{cris}(A_0/S_0)_{S, \gothp_i,j}$, and that its base change $\scrH_{\gothp_i,j, S_0}^{(1)}$ to $S_0$ contains the subbundle $\scrF_{\gothp_i,j}^{(1)}$ by construction. The set of possible choices for a rank one $\calO_{S}$-subbundle $\tilde \scrF_{\gothp_i,j}^{(1)}$ of $\scrH_{\gothp_i,j}^{(1)}$ lifting $\scrF_{\gothp_i,j}^{(1)}$ is a torsor under
\[
\calH om_{\calO_{S_0}} \big (\scrF_{\gothp_i,j}^{(1)}, \scrH_{\gothp_i,j, S_0}^{(1)} / \scrF_{\gothp_i,j}^{(1)}\big) \otimes_{\calO_{S_0}} \calI
\cong (\scrF_{\gothp_i,j}^{(1)})^{\otimes -2} \otimes_{\calO_{S_0}} (\wedge^2_{\calO_{S_0}} \scrH_{\gothp_i,j, S_0}^{(1)}) \otimes_{\calO_{S_0}} \calI.
\]
This proves the part of the theorem regarding $\tau_{\gothp_i,j}^1$.

Now suppose that we have lifted $\scrF_{\gothp_i,j}^{(l-1)}$ to $\tilde \scrF_{\gothp_i,j}^{(l-1)}$ and that we want to lift $\scrF_{\gothp_i,j}^{(l)}$ to an $\calO_F$-stable subsheaf  $\tilde \scrF_{\gothp_i,j}^{(l)}$ of $\calH^1_\mathrm{cris}(A_0/S_0)_{S, \gothp_i,j}$ so that
\begin{itemize}
\item it contains $\tilde \scrF_{\gothp_i,j}^{(l-1)}$,
\item it is locally a direct summand of $\calH^1_\mathrm{cris}(A_0/S_0)_{S, \gothp_i,j}$,
\item $\tilde \scrF_{\gothp_i,j}^{(l)} /  \tilde \scrF_{\gothp_i,j}^{(l-1)}$  is locally free of rank one over $\calO_S$,
\item $\calO_F$ acts on  $\tilde \scrF_{\gothp_i,j}^{(l)} /  \tilde \scrF_{\gothp_i,j}^{(l-1)}$ via $\tau_{\gothp_i,j}^{l}$, and
\item
$\tilde \scrF_{\gothp_i,j}^{(l)} \subseteq (\tilde \scrF_{\gothp_i,j}^{(l-1)})^\perp$ (and hence $\tilde \scrF_{\gothp_i,j}^{(l)} \subseteq (\tilde \scrF_{\gothp_i,j}^{(l)})^\perp$ because $\langle \cdot,\cdot\rangle$ satisfies condition \eqref{E:strongly symplectic}).
\end{itemize} 
Define
\[
\scrH_{\gothp_i,j}^{(l)}: = \Big\{ z \in (\tilde \scrF_{\gothp_i,j}^{(l-1)})^\perp \big/ \tilde \scrF_{\gothp_i,j}^{(l-1)}\; \Big|\;
[\varpi_i]z - \tau_{\gothp_i, j}^{l}(\varpi_i)z  =0 \Big\}.
\]
\textbf{Claim:} We assert that the following two facts hold: 
\begin{enumerate}
 \item the sheaf $\scrH_{\gothp_i,j}^{(l)}$ is a rank-two $\calO_S$-subbundle of $\calH^1_\cris(A_0/S_0)_{S,\gothp_i,j} \big/ \tilde \scrF_{\gothp_i,j}^{(l-1)}$ (cf. \cite[Proposition 5.2]{pappas-rapoport}), and
 \item $\wedge^2_{\calO_S} \scrH_{\gothp_i,j}^{(l)} \cong \big(\wedge^2_{\calO_F \otimes_\ZZ \calO_S}\calH^1_\mathrm{cris}(A_0/S_0)_S\big)  \otimes_{\calO_F \otimes_\ZZ \calO_S, \tau_{\gothp_i,j}^l \otimes 1} \calO_{S}$.
\end{enumerate}

\noindent Granting the claim, we see that constructing the desired lift $\tilde \scrF_{\gothp_i,j}^{(l)}$ is equivalent to lifting $\scrF_{\gothp_i,j}^{(l)} / \scrF_{\gothp_i,j}^{(l-1)}\subset\scrH_{\gothp_i,j,S_0}^{(l)}$to a rank one $\calO_S$-subbundle of $\scrH_{\gothp_i,j}^{(l)}$. The set of such lifts is a torsor under
\begin{align*}
\calH om_{\calO_{S_0}} &\big (\scrF_{\gothp_i,j}^{(l)} / \scrF_{\gothp_i,j}^{(l-1)}, \scrH_{\gothp_i,j, S_0}^{(l)} / 
(\scrF_{\gothp_i,j}^{(l)}/\scrF_{\gothp_i,j}^{(l-1)})\big) \otimes_{\calO_{S_0}} \calI
\\
&
\cong \big( \scrF_{\gothp_i,j}^{(l)} / \scrF_{\gothp_i, j}^{(l-1)}\big)^{\otimes -2}\otimes_{\calO_{S_0}}
(\wedge^2_{\calO_{S_0}} \scrH_{\gothp_i,j,S_0}^{(l)})
\otimes_{\calO_{S_0}} \calI
\\
&\cong \big( \scrF_{\gothp_i,j}^{(l)} / \scrF_{\gothp_i, j}^{(l-1)}\big)^{\otimes -2}\otimes_{\calO_{S_0}}
\big( (\wedge^2_{\calO_F \otimes_\ZZ \calO_{S_0}}\calH^1_\mathrm{dR}(A_0/S_0)) \otimes_{\calO_F \otimes_\ZZ \calO_{S_0},\tau_{\gothp_i,j}^l \otimes 1} \calO_{S_0}\big)
\otimes_{\calO_{S_0}} \calI.
\end{align*}
The theorem follows. The Claim is proved below.
\end{proof}
\noindent \emph{Proof of the Claim.} We first observe that the statements in the Claim are stable under base change in $S$. Moreover they are purely statements about $\calH^1_\mathrm{cris}(A_0/S_0)_{S, \gothp_i,j}$ as a locally free module of rank two over 
$\calO_{F_{\gothp_i}}\otimes_{W(\FF_{\gothp_i}),\tau_{\gothp_i,j}}\calO_S $, endowed with a non-degenerate symplectic pairing satisfying \eqref{E:OF hermitian} and \eqref{E:strongly symplectic}. In particular, we can prove the Claim for abstract locally free sheaves with the above properties.

To prove (1) it suffices to pass to the completion of the stalks of $\scrH_{\gothp_i,j}^{(l)}$ at each point of $S$. We can therefore assume that $S$ is the spectrum of a complete Noetherian local ring, hence the spectrum of a quotient of a regular complete Noetherian local ring $L$.
By induction on $l$ we can lift the sheaves involved in the definition of $\scrH_{\gothp_i,j}^{(l)}$ to $\Spec L$ (and keeping the condition~\eqref{E:strongly symplectic}). We can then assume that $L$ is reduced, 
and finally that it is a field\footnote{In the proof of \cite[Proposition~5.2(b)]{pappas-rapoport} the authors claim right away that it suffices to prove part (1) of the Claim when $S$ is the spectrum of a field; we think one has to slightly waggle the argument when $S$ is not reduced.}.
If $L$ is a field of characteristic zero, statement (1) is clear as $\calH^1_\mathrm{cris}(A_0/S_0)_{S, \gothp_i,j}$ is isomorphic to
\[
(\calO_{F_{\gothp_i}} \otimes_{W(\FF_{\gothp_i}), \tau_{\gothp_i,j}}L)^{\oplus 2}
\cong 
\bigoplus_{\alpha=1}^{e_i} (\calO_{F_{\gothp_i}} \otimes_{\calO_{F_{\gothp_i}}, \tau_{\gothp_i,j}^{\alpha}}L)^{\oplus 2},
\]
and hence
\[
 \scrH_{\gothp_i,j}^{(l)} \cong (\calO_{F_{\gothp_i}} \otimes_{\calO_{F_{\gothp_i}}, \tau_{\gothp_i,j}^{l}}L)^{\oplus 2}\cong L^{\oplus 2}.
\]
If $L$ is a field of characteristic $p$ then $\calH^1_\mathrm{cris}(A_0/S_0)_{S, \gothp_i,j}$ is isomorphic to
\begin{equation}
\label{E:special fiber crystalline cohomology}
 \frac{L [x]}{x^{e_i}L[x]}\oplus \frac{L [x]}{x^{e_i}L[x]},
\end{equation}
where the action of $\varpi_i$ is given by multiplication by $x$.
We need to show that for any non-zero isotropic $L[x]$-submodule $W$ of \eqref{E:special fiber crystalline cohomology}
of $L$-dimension $<e_i$, the quotient $W^\perp/W$ has two-dimensional $x$-torsion.
We may pick a basis so that $W \cong x^aL[x]/x^{e_i}L[x] \oplus x^bL[x]/x^{e_i}L[x]$ for integers $a,b$ with $0 < a,b \leq  e_i$ and $a+b >e_i$.
Using \eqref{E:OF hermitian} and \eqref{E:strongly symplectic}, we see that $W^\perp$ contains $x^{e_i-b} L[x]/x^{e_i}L[x] \oplus x^{e_i-a}L[x]/x^{e_i}L[x]$ (under the same basis).
Since the pairing is non-degenerate, dimension considerations force this
containment to be an equality.
So under the condition $a+b>e_i$, we explicitly see that $W^\perp/W$ has two-dimensional $x$-torsion, namely, $x^{a-1}L[x]/x^aL[x] \oplus x^{b-1}L[x]/x^bL[x]$.

We now turn to the proof of statement (2) of the Claim, but we first introduce some notation. For an integer $l'$ such that $0\leq l' \leq e_i$ we set
\[
P^{\leq l'}(x) := \prod_{1\leq \alpha \leq l'} (x - \tau_{\gothp_i, j}^{\alpha}(\varpi_i)) \in \calO[x],
\]
with the convention that a product over the empty set is equal to $1$.
Recall that $\calO_{F_{\gothp_i}} \otimes_{W(\FF_{\gothp_i}), \tau_{\gothp_i,j}} \calO_S \cong \calO_S[x] / (E_{\gothp_i,j}(x))$ (where
the action of $\varpi_i$ is given by multiplication by $x$), so that
\[
T:=
\{z \in \calH^1_\cris(A_0/S_0)_{S, \gothp_i,j}\;|\; P^{\leq l}([\varpi_i])\cdot z = 0 \}
\]
is an $\calO_S[x] /(P^{\leq l}(x))$-module
locally free
of rank two.
We denote by $\tilde H$ the preimage of $\scrH_{\gothp_i,j}^{(l)}$ in $\calH^1_\cris(A_0/S_0)_{S, \gothp_i,j}$.

The inclusion $\tilde H \subset T$ induces a homomorphism of $\calO_S$-modules: 
\begin{equation}
\label{E:KS lifting map}
\varphi:
\wedge^2_{\calO_S}\tilde H \longto
\wedge^2_{\calO_S[x] /(P^{\leq l}(x))} T.
\end{equation}
We claim that $\varphi(y \wedge z)=0$ for $y \in \tilde H$ and $z \in \tilde \scrF_{\gothp_i,j}^{(l-1)}$, and that the image of $\varphi$ equals
\begin{equation}
\label{E:what wedge 2 H equals to}
P^{\leq l-1}(x)\cdot \wedge^2_{\calO_S[x] /(P^{\leq l}(x))} T.
\end{equation}
Granting this, we deduce that the invertible $\calO_S$-module
\[
\wedge^2_{\calO_S}\scrH_{\gothp_i,j}^{(l)} \cong \wedge^2_{\calO_S}(\tilde H/\tilde \scrF_{\gothp_i,j}^{(l-1)})
\]
surjects onto, and hence is isomorphic to, the invertible $\calO_S$-module \eqref{E:what wedge 2 H equals to}. Since the latter is canonically isomorphic to $\big(\wedge^2_{\calO_F \otimes_\ZZ \calO_S}\calH^1_\dR(A/S)\big) 
 \otimes_{\calO_F \otimes_\ZZ \calO_S, \tau_{\gothp_i,j}^l\otimes 1} \calO_S $, statement (2) of the Claim follows.

We are left with proving the two statements granted above. In order to do so, as before, we can assume that $S$ is the spectrum of a field $L$.
If $L$ has characteristic zero, the two statements are obvious, so we can assume that $L$ has  characteristic $p$. In this case we have
$T \cong \big( L[x]/ x^lL[x]\big)^{\oplus 2}$, and $\tilde \scrF_{\gothp_i,j}^{(l-1)} \cong (x^{l-a}L[x]/x^lL[x]) \oplus (x^{l-b}L[x]/x^lL[x])$ for some non-negative integers $a,b$ such that $a+b = l-1$.  Hence,
\[
\tilde H \cong
\frac{x^{l-a-1}L[x]}{x^lL[x]} \oplus \frac{x^{l-b-1}L[x]}{x^lL[x]}.
\]
The map $\varphi:\wedge^2_L\tilde H \to \wedge^2_{L[x]/(x^l)}(L[x]/x^lL[x])^{\oplus2}=L[x]/x^lL[x]$ is therefore given by:
\[
(A\bar x^{l-a-1}\oplus B\bar x^{l-b-1})\wedge_L(C\bar x^{l-a-1}\oplus D\bar x^{l-b-1}) \mapsto (AD-BC)\bar x^{l-1}, 
\]
where $\bar x$ denotes the image of $x$ in $L[x]/x^lL[x]$, and $A,B,C,D\in L[x]/x^lL[x]$. 
It is now clear that $\varphi$ vanishes on those elements for which both $C$ and $D$ are divisible by $\bar x$. These are the elements
belonging to $\tilde H \wedge_L \tilde\scrF_{\gothp_i,j}^{(l-1)}$. It also follows that $\varphi(\wedge^2_L \tilde H) = x^{l-1}L[x] / x^lL[x]$. This completes the proof.
\hfill \qed

The proof of Theorem~\ref{T:Kodaira-Spencer ramified}, especially the Claim implies the following
\begin{corollary}
\label{C:de Rham cohomology at tau}
For the universal filtration $\underline \calF = (\calF_{\gothp_i,j}^{(l)})$ on $\omega_{\calA^\PR/\calM^\PR}$, we put
\[
\calH_{\gothp_i,j}^{(l)}: = \big\{ z \in (\calF_{\gothp_i,j}^{(l-1)})^\perp / \calF_{\gothp_i,j}^{(l-1)} \,  \big| \, [\varpi_i]z-\tau_{\gothp_i,j}^l(\varpi_i)z=0\big\}.
\]
Then it is a subbundle of $H^1_\dR(\calA^\PR/\calM^\PR)_{\gothp_i,j} / \calF_{\gothp_i,j}^{(l-1)}$ of rank two, containing $\calF_{\gothp_i,j}^{(l)} / \calF_{\gothp_i,j}^{(l-1)}$. Moreover, we have a canonical isomorphism
\[
\wedge^2_{\calO_{\calM^\PR}} \calH_{\gothp_i,j}^{(l)} \cong \dot \epsilon_{\tau_{\gothp_i,j}^l}
\]
\end{corollary}

These $\calH_{\gothp_i,j}^{(l)}$ will serve the role of ``de Rham cohomology at $\tau_{\gothp_i,j}^l$" when defining the partial Hasse invariants, analogous to the unramified case.

\subsection{Toroidal compactification}
\label{S:toroidal compactification}
For any ideal class $\mathfrak{c}\in\mathfrak{C}$ fix a rational polyhedral
admissible cone decomposition $\Phi_{\mathfrak{c}}$ for each isomorphism
class of $\Gamma_{00}({\mathcal{N}})$-cusps of the $\calO$-scheme $\mathcal{M}_{\mathcal{\mathfrak{c}}}^\Ra$
(\cite[5.1]{DT}). By \textit{loc.cit.}, Th\'{e}or\`{e}me 5.2, there exists a
smooth scheme $\mathcal{M}_{\mathcal{\mathfrak{c}},\Phi_{\mathfrak{c}}%
}^{\operatorname*{Ra, tor}}$ over $\calO$ containing $\mathcal{M}^\Ra
_{\mathcal{\mathfrak{c}}}$ as a fiberwise dense open subscheme.
Moreover, the action of $\calO_F^{\times, +} / (\calO_{F, {\mathcal{N}}}^\times)^2$ on the boundary divisor $\dot \ttD_\gothc: = \mathcal{M}_{\mathfrak{c},\Phi_{\mathfrak{c}}
}^{\operatorname*{Ra, tor}} - \mathcal{M}_{\mathfrak{c}}^{\operatorname*{Ra}}$ is free by construction \cite[Th\'eor\`em~7.2]{dimitrov}.

Denote by $\calM_\gothc^{\PR,\tor}$ (resp. $\calM_\gothc^{\DP, \tor}$) the scheme obtained by gluing $\calM^{\Ra, \tor}_{\gothc, \Phi_\gothc}$
to $\calM_\gothc^\PR$ (resp. $\calM_\gothc^\DP$) over $\calM_\gothc^\Ra$. (Here and later, we  shall
drop the subscript $\Phi_\gothc$ for simplicity). The scheme $\calM_\gothc^{\PR, \tor}$ is proper and smooth over $\Spec \calO$.  
We set $\mathcal{M}^{\operatorname*{?,tor}}:=%
{\textstyle\coprod\nolimits_{\mathfrak{c}\in\mathfrak{C}}}
\mathcal{M}_{\mathcal{\mathfrak{c}}}^{\operatorname*{?, tor}}$ for $?\in\{\DP, \PR, \Ra$\}.
The
boundary $\dot \ttD:=\mathcal{M}^{?,\operatorname*{tor}}-\mathcal{M}^?$ is a
relative simple normal crossing divisor on $\mathcal{M}^{?,\operatorname*{tor}}$.

Let $\Sh^{?, \tor}$ denote the quotient of $\calM^{?,\tor}$ by the action of the group $\calO_F^{\times, +} / (\calO_{F, {\mathcal{N}}}^\times)^2$.
Put $\ttD: = \Sh^{?,\tor} - \Sh^?$; it is the quotient of $\dot \ttD$ and it is a divisor with simple normal crossings.

There exists a semi-abelian scheme $\pi:\mathcal{A}^{\operatorname*{?,tor}%
}\rightarrow\mathcal{M}^{\operatorname*{?,tor}}$ extending the universal abelian
scheme $\mathcal{A}^?\rightarrow\mathcal{M}^?$; it is endowed with an
$\mathcal{O}_{F}$-action and an embedding $\mu_{{\mathcal{N}}}\rightarrow\mathcal{A}^{\operatorname*{?,tor}}$ extending the
corresponding data on $\mathcal{A}^?$.
If $e:\calM^{\operatorname*{\Ra, \tor}}\rightarrow\mathcal{A}
^{\operatorname*{\Ra,tor}}$ denotes the unit section of the semi-abelian scheme
$\mathcal{A}^{\operatorname*{Ra, tor}}$ over $\mathcal{M}%
^{\operatorname*{Ra, tor}}$, we set
\[
\dot \omega^{\Ra,\tor}:=e^{\ast}\Omega_{{\mathcal{A}}^{\operatorname*{Ra,tor}}/\mathcal{M}^{\operatorname*{Ra,tor}}}^{1}.
\]
It is a locally free $(\calO_F \otimes_\ZZ \calO_{\calM^{\Ra, \tor}})$-module of rank one over $\calM^{\Ra, \tor}$.
For $\tau \in \Sigma$ we set:
\[
\dot \omega^{\Ra, \tor}_\tau :=
\dot \omega^{\Ra, \tor}_\tau \otimes_{\calO_F \otimes_\ZZ \calO_{\calM^{\Ra, \tor}}, \tau \otimes 1} \calO_{\calM^{\Ra, \tor}}.
\]
The sheaf $\dot \omega^{\Ra, \tor}_\tau$ agrees with the sheaf $\dot \omega_\tau$ introduced in Notation \ref{N:M Sh A} when they are both restricted to $\calM^\Ra$.
Gluing $\dot \omega_\tau$ with $\dot \omega_\tau^{\Ra, \tor}$ over $\calM^\Ra$ we obtain a line bundle $\dot\omega_\tau^\tor$ over $\calM^{\PR, \tor}$.
To lighten the load on notation, we will later simply write $\dot \omega_\tau$ for $\dot \omega_\tau^\tor$ when no confusion arises.
The sheaf $\dot \omega_\tau$ carries an action of $\calO_F^{\times, +}$ as described in Notation~\ref{N:M Sh A}.

Similarly, the trivial line bundle $\dot\varepsilon_\tau$ on $\calM^{\PR}$ extends to a line bundle
\[
\dot\varepsilon_\tau \cong (\gothc \gothd^{-1}\otimes_\ZZ \calO_{\calM^{\PR, \tor}} ) \otimes_{\calO_F\otimes_\ZZ\calO_{\calM^{\PR, \tor}}  ,  \tau \otimes1} \calO _{\calM^{\PR, \tor}}
\]
on $\calM^{\PR,\tor}$ carrying a natural action of $\calO_F^{\times, +} / (\calO_{F,{\mathcal{N}}}^\times)^2$ as described in Notation~\ref{N:M Sh A}.

The line bundles $\dot\omega_\tau$ and $\dot\varepsilon_\tau$ descend to line bundles over $\Sh^{\PR, \tor}$, which we denote by $\omega_\tau$ and $\varepsilon_\tau$ respectively.  We warn the reader that the line bundle $\varepsilon_\tau$ may not be trivial over $\Sh^{\PR, \tor}$.

\subsection{Geometric Hilbert modular forms}
\label{S:geometric HMF}
A \emph{paritious weight} $\kappa$ is a tuple $((k_\tau)_{\tau \in \Sigma}, w) \in \ZZ^{\Sigma} \times \ZZ$ such that $k_\tau \equiv w \pmod 2$ for every $\tau \in \Sigma$.
We say that $\kappa$ is \emph{regular} if moreover $k_\tau >1$ for all $\tau\in\Sigma$.

For $\kappa = ((k_\tau)_{\tau \in \Sigma}, w)$ a paritious weight, we define
\[
\dot \omega^\kappa: = \bigotimes_{\tau \in \Sigma} \big(\dot  \omega_{\tau}^{\otimes k_\tau} \otimes_{\calO_{\calM^{\PR, \tor}}}\dot  \varepsilon_{\tau}^{\otimes (w-k_\tau)/2} \big), \textrm{ and }\omega^\kappa: = \bigotimes_{\tau \in \Sigma} \big( \omega^{\otimes k_\tau}_{\tau} \otimes_{\calO_{\Sh^{\PR, \tor}}} \varepsilon_{\tau}^{\otimes (w-k_\tau)/2} \big).
\]
They are line bundles over $\calM^{\PR, \tor}$ and $\Sh^{\PR, \tor}$, respectively.  We remind the reader that these line bundles depend on the fixed choice of an ordering of the $\tau_{\gothp_i, j}^l$'s.

A \emph{(geometric) Hilbert modular form} over a Noetherian $\calO$-algebra $R$, of level $\Gamma_{00}({\mathcal{N}})$ and (paritious) weight $\kappa$ is an element of the finite $R$-module $H^0(\Sh_R^{\PR, \tor}, \omega^\kappa_R)$,
where the subscript $R$ indicates base change to $R$ over $\calO$.
Such a form is called \emph{cuspidal} if it belongs to the submodule $H^0(\Sh_R^{\PR, \tor}, \omega^\kappa_R(-\ttD))$.
By the K\"ocher principle (\cite[Th\'eor\`em~7.1]{DT}), if $[F:\QQ]>1$, we have
\[
H^0(\calM_R^{\PR, \tor}, \dot \omega_R^\kappa) = H^0(\calM_R^{\PR}, \dot \omega_R^\kappa), \textrm{ and hence } H^0(\Sh_R^{\PR, \tor}, \omega_R^{\kappa}) = H^0(\Sh_R^{\PR}, \omega^\kappa_R).
\]
In particular, the space of geometric Hilbert modular forms is independent on the choice of toroidal compactification that we have made.
Clearly, $H^0(\Sh_R^{\PR,\tor}, \omega^\kappa_R)$ is a direct sum of $H^0(\Sh_{\gothc, R}^{\PR,\tor}, \omega^\kappa_R)$ over all $\gothc \in\gothC$; we call elements in such a direct summand \emph{$\gothc$-polarized Hilbert modular forms}.

Following Katz (\cite{Katz}), we can describe ($\gothc$-polarized) Hilbert modular forms as follows.
Let $R'$ be a Noetherian $R$-algebra and let $\gothc \in \gothC$. A \emph{($\gothc$-polarized) test object} over $R'$ is a tuple $(A, \lambda, i, \underline \scrF, \underline \eta, \underline \xi)$ consisting of a $\gothc$-polarized HBAS $(A, \lambda, i, \underline \scrF)$ defined over $R'$ and endowed with a $\Gamma_{00}({\mathcal{N}})$-level structure $i$ and a filtration $\underline\scrF$ as in Subsection~\ref{S:moduli HBAS}, together with a choice $\underline \eta = (\eta_\tau)_{\tau \in \Sigma}$ (resp. $\underline \xi = (\xi_\tau)_{\tau \in \Sigma}$) of generators $\eta_{\tau}$ (resp. $\xi_{\tau}$) for each free rank one $R'$-module $\dot\omega_{A/R',\tau}$ (resp. $\dot\varepsilon_{A/R',\tau}$). (Recall that $\dot\omega_{A/R',\tau}$ is one of the subquotients of the filtration $\underline\scrF$).

A $\gothc$-polarized Hilbert modular form over $R$ of level $\Gamma_{00}({\mathcal{N}})$ and weight $\kappa$ can be interpreted
as a rule $f$ which assigns to any Noetherian $R$-algebra $R^{\prime}$ and to any $\gothc$-polarized test
object $(A,\lambda,i,\underline \scrF,\underline \eta, \underline \xi)$ over $R^{\prime}$ an element $f(A,\lambda
,i,\underline \scrF,\underline \eta, \underline \xi)\in R^{\prime}$ in such a way that
\begin{itemize}
\item[(i)]
this assignment depends only on the
isomorphism class of $(A,\lambda,i,\underline \scrF,\underline \eta, \underline \xi)$,
\item[(ii)]
is compatible with
base change in $R'$,
\item[(iii)]
satisfies
$f(A,u\lambda
,i,\underline \scrF,\underline \eta, \underline \xi) = f(A,\lambda
,i,\underline \scrF,\underline \eta, \underline \xi)$ for any $u \in\calO_F^{\times,+}$, and
\item[(iv)]
satisfies
\[
f(A,\lambda,i,\underline \scrF,\underline a\underline \eta, \underline b\underline \xi)=\prod_{\tau \in \Sigma} a_\tau^{-k_\tau}b_\tau^{-(w-k_\tau)/2}\cdot f(A,\lambda,i,\underline \scrF,\underline \eta, \underline \xi),
\]
for all $\underline a = (a_\tau)_{\tau \in \Sigma} \in (R'^{\times})^\Sigma$ and all $\underline b = (b_\tau)_{\tau \in \Sigma} \in (R'^{\times})^\Sigma$, where $\underline a \underline \eta: = (a_\tau \eta_\tau)_{\tau \in \Sigma}$ and $\underline b \underline \xi: = (b_\tau \xi_\tau)_{\tau \in \Sigma}$.
\end{itemize}

\begin{remark}
The above geometric interpretation as ``Katz
modular forms" can also be given to sections of non-normalized weight sheaves, \emph{i.e.}, sections of
$\otimes_{\tau \in\Sigma} \omega_{\tau, R}^{\otimes k_\tau} \otimes_{\Sh_R^{\PR, \tor}} \varepsilon_{\tau, R}^{\otimes n_\tau}$,  where $k_\tau$ and $n_\tau$ are any integers (without the additional restriction
that $k_\tau + 2n_\tau$ is constant with respect to $\tau$). When $R$ has characteristic zero though, any form of non-normalized weight is necessarily zero (due to condition (iii) above), while when $R$ has characteristic $p$, non-zero forms of non-normalized weights do exist (for example the generalized partial Hasse invariants we construct later).
\end{remark}

\subsection{Tame Hecke operators}
\label{S:tame hecke operator}
Since the final application of our paper is to associate Galois representations to Hecke eigenclasses of the coherent cohomology $H^*(\Sh^\PR, \omega^\kappa)$, we now explain how the tame Hecke algebra acts on this cohomology group.

Fix a paritious weight $\kappa = ((k_\tau)_{\tau \in \Sigma}, w)$, and fix an ideal $\mathfrak{a} \subset \calO_F$ coprime to $p$.
Denote by $\mathcal{M}(
\mathfrak{a})^\PR  $ the $\calO$-scheme representing the functor that takes a
locally Noetherian $\calO$-scheme $S$ to the set of isomorphism classes of tuples $(A,\lambda,i, \underline \scrF;C)$ consisting of an $S$-point
$(A,\lambda,i, \underline \scrF)$ of $\calM^\PR$ together with an $\mathcal{O}_{F}$-stable closed
subgroup $S$-scheme $C$ of $A$ such that

\begin{description}
\item[C1] $i(\mu_{{\mathcal{N}}})$ is disjoint from $C$, and

\item[C2] \'{e}tale locally on $S$, the group scheme $C$ is $\mathcal{O}_{F}%
$-linearly isomorphic to the constant group-scheme $\mathcal{O}_{F}%
/\mathfrak{a}$.
\end{description}

The group $\calO_F^{\times,+} / (\calO_{F, {\mathcal{N}}}^\times)^2$ acts freely on $\calM(\gotha)^\PR$; we denote by $\Sh(\gotha)^\PR$ the corresponding quotient.
The construction described in \cite[2.2.5]{emerton-reduzzi-xiao}, suitably modified to accommodate the presence of the filtrations $\underline \scrF$, gives rise to natural finite and \'etale maps\footnote{The maps $\pi_1$ and $\pi_2$ are finite and \'etale because $A[\gotha]$ is \'etale.}
\[\pi_{1}:\Sh(
\mathfrak{a})^\PR  \rightarrow\Sh^\PR \textrm{  and  } \pi_{2}:\Sh(
\mathfrak{a})^\PR  \rightarrow\Sh^\PR
\] 
defined respectively by forgetting the subgroup $C$ and by
quotienting by it. (We remark that, as in \cite{emerton-reduzzi-xiao}, one first defines morphisms $\calM(\gotha)^\PR\to \calM^\PR$, and then passes to the quotient). In particular, the filtration on 
the image of an $S$-point $(A,\lambda,i,\underline\scrF;C)$ of $\mathcal{M}(\mathfrak{a})^\PR$
under $\pi_2$ is given by $\pi_*^0(\underline \scrF)$, where $\pi_*^0$ denotes the inverse of the $(\calO_F \otimes_\ZZ \calO_S)$-linear isomorphism of cotangent sheaves at the origin $\omega_{(A/C)/S}\to\omega_{A/S}$ induced by the \'{e}tale quotient map 
$\pi:A \to A/C$.

As in \textit{loc.cit.}, we obtain a morphism of sheaves
\[ 
T_{\mathfrak{a}}:\pi_{2}^{\ast
}\omega^{\kappa}\rightarrow\pi_{1}^{\ast}\omega^{\kappa}.
\]
(Note that the Kodaira-Spencer isomorphism used in \textit{loc.cit} is replaced here by Theorem~\ref{T:Kodaira-Spencer ramified}).
Applying $\pi_{1*}$ to the above morphism, taking the trace, and taking $H^j$ for $j \geq 0$, we obtain the action of the Hecke operator $T_\gotha$ on the cohomology:
\[
H^j(\Sh^{\PR}, \omega^\kappa) \to H^j(\Sh(\gotha)^{\PR}, \pi_2^*\omega^\kappa) \to H^j(\Sh^{\PR}, \pi_{1*} \pi_2^* \omega^\kappa) \xrightarrow{\pi_{1*}(T_\gotha)} H^j(\Sh^{\PR}, \omega^\kappa).
\]
The morphism $T_{\mathfrak{a}%
}$ extends to $\Sh^{\PR, \tor}$, and we will denote its action on $H^{j}(\Sh^{\PR, \tor},\omega^{\kappa})$
again by $T_\mathfrak{a}$.

If $\kappa^{\prime}$ is another paritious weight and $\zeta:\omega^{\kappa}%
\rightarrow\omega^{\kappa^{\prime}}$ is a homomorphism, we say that $\zeta$ is
\emph{equivariant with respect to the action of the Hecke operator
}$T_{\mathfrak{a}}$ if the following diagram commutes:
\[
\xymatrix{
\pi_2^*\omega^\kappa \ar[r]^{T_\gotha} \ar[d]^{\pi_2^*\zeta} & \pi_1^*\omega^\kappa\ar[d]^{\pi_1^*\zeta}
\\
\pi_2^*\omega^{\kappa'} \ar[r]^{T_\gotha}  & \pi_1^*\omega^{\kappa'}.
}
\]

For later use, we now describe the action of the Hecke operator
$T_{\mathfrak{a}}$ on Hilbert modular forms using test objects (cf.
\cite{Hida}, 4.2.9).
We fix $\gothc \in \gothC$.
Let $(A,\lambda,i,\underline \scrF, \underline \eta, \underline \xi)$ be a $\mathfrak{c}$-polarized test object as in
\ref{S:geometric HMF}, defined over a Noetherian $\calO$-algebra $R$. Fix an $\mathcal{O}_{F}
$-stable closed subgroup scheme $C$ of $A$ which is defined over $R$
and satisfies conditions \textbf{C1} and \textbf{C2} given above.
The corresponding isogeny of abelian schemes $\pi:A\rightarrow A^{\prime
}:=A/C$ is \'{e}tale. We let $(A^{\prime},\pi_{\ast}\lambda,\pi_{\ast}i)$ be
the $\mathfrak{ca}$-polarized HBAS obtained by quotienting $(A,\lambda,i)$ by
$C$. Since $\pi$ is an \'{e}tale isogeny there is a canonical isomorphism
$\pi^{\ast}\Omega^1_{A^{\prime}/R}\cong\Omega^1_{A/R}$ whose
inverse induces $(\mathcal{O}_{F}\otimes_\ZZ R)$-linear identifications 
\[
\pi_{\ast}^{0}: \omega_{A / R} \to \omega_{A'/R}, \textrm{  and  } 
\pi_{\ast}^{0}: \wedge^2_{\mathcal{O}_{F}\otimes_\ZZ R}\calH_{\dR}^1(A / R) \to \wedge^2_{\mathcal{O}_{F}\otimes_\ZZ R}\calH_{\dR}^1(A' / R).
\]
Let $\mathfrak{c}^{\prime}$ be the unique fractional ideal in
$\mathfrak{C}$ for which there is an $\mathcal{O}_{F}$-linear isomorphism
$\theta: \mathfrak{ca}\cong\mathfrak{c}^{\prime}$ preserving the positive cones on
both sides, and let $f\in H^{0}(\Sh^\PR_{\mathfrak{c}^{\prime}}%
,\omega^{\kappa})$ be a $\mathfrak{c}^{\prime}$-polarized Hilbert modular
form. For any $\calO$-algebra $R$ and any $\mathfrak{c}$-polarized test
object $(A,\lambda,i,\underline \scrF, \underline \eta, \underline \xi)$  defined over $R$ we have:
\begin{equation}
\label{E:tame hecke action}
\left(T_{\mathfrak{a}}f\right)  (A,\lambda,i,\underline \scrF, \underline \eta,\underline \xi)=\frac{1}
{\operatorname*{Nm}\nolimits_{%
\mathbb{Q}
}^{F}\left(  \mathfrak{a}\right)  }%
{\displaystyle\sum\nolimits_{C}}
f(A/C,\pi_{\ast}\lambda,\pi_{\ast}i,\pi_*^0\underline \scrF, \pi_{\ast}^{0}\underline \eta, \pi_*^0 \underline \xi),
\end{equation}
where $C$ varies over the closed $\mathcal{O}_{F}$-stable subgroups
of $A$ satisfying conditions \textbf{C1} and \textbf{C2}.
This expression does not depend on the choice of the isomorphism $\theta$.

\begin{notation}\label{N:universal hecke}
Let $\ttS$ denote  a finite set of places of $F$ containing the places dividing
$p{\mathcal{N}}$ and the archimedean places.
The polynomial ring
\[
\TT_{\mathtt{S}}^{\univ}:=\calO[t_\gothq; \gothq \textrm{ a place of }F \textrm{ not in }\ttS]
\]
is called the \emph{universal Hecke algebra}. It acts on the cohomology groups 
$H^{j}(\Sh^{\PR,\tor},\omega^{\kappa})$ and $H^{j}(\Sh^{\PR,\tor},\omega^{\kappa}(-\ttD))$ via $t_\mathfrak{q}\mapsto T_\mathfrak{q}$.

\end{notation}

\section{Generalized partial Hasse invariants on splitting models}

Unlike in the case when $p$ is unramified in $F$, when $p$ ramifies there are fewer partial Hasse invariants available (cf. \cite{AG}).  To remedy this, we construct new invariants arising from the action of $\calO_F$ on the universal filtration over the splitting model; we obtain in this way a good stratification of the special fiber of $\calM^\PR$.
\subsection{Factorization of the Verschiebung map}
\label{S:generalized Hasse}
From now on, if no confusion arises, we will drop the superscript $\PR$ appearing in the schemes introduced in the previous sections. In particular, 
we set $\calA : = \calA^\PR$, $\calM: = \calM^\PR$, and $\Sh: = \Sh^\PR$. These are schemes over $\calO$, and we denote their special fibers by $\calA_\FF$, $\calM_\FF$, and $\Sh_\FF$ respectively.

Over $\calM_\FF$, the sheaf $\calH^1_\dR(\calA_\FF/ \calM_\FF)_{\gothp_i,j}$ is a locally free
$\calO_{\calM_\FF}[x] / (x^{e_i})$-module of rank two, where the action of $\varpi_i$ is given by multiplication by $x$.  Accordingly, each subquotient $\calF_{\gothp_i,j}^{(l)} / \calF_{\gothp_i,j}^{(l-1)}$ is annihilated by $x$. Recall from Corollary~\ref{C:de Rham cohomology at tau} that
\[
\calH_{\gothp_i,j}^{(l)} : = \big\{ z \in (\calF_{\gothp_i,j}^{(l-1)})^\perp / \calF_{\gothp_i,j}^{(l-1)} \, \big|\, x \cdot z = 0\big\}
\]
is a rank two subbundle of $\calH^1_\dR(\calA_\FF/\calM_\FF)_{\gothp_i,j} / \calF_{\gothp_i,j}^{(l-1)}$ over $\calM_\FF$, containing $\calF_{\gothp_i,j}^{(l)} / \calF_{\gothp_i,j}^{(l-1)}$. 

\begin{remark}
In contrast to what happens over the integral model, over $\calM_\FF$ the line bundles
\[
\wedge^2_{\calO_{\calM_\FF}} \calH_{\gothp_i,j}^{(l)} \cong 
\dot \varepsilon_{\tau_{\gothp_i,j}^l} =
\wedge^2_{\calO_{\calM_\FF}} \big(\calH^1_\dR(\calA_\FF/\calM_\FF)_{\gothp_i, j} / [\varpi_i]\cdot \calH^1_\dR(\calA_\FF/\calM_\FF)_{\gothp_i, j} \big)
\]
for each $i$ and $j$ are canonically independent from $l$.
Moreover, we have canonical isomorphisms $\dot \varepsilon_{\tau_{\gothp_i, j}^1} \cong \dot \varepsilon_{\tau_{\gothp_i, j-1}^{e_i}}^{\otimes p}$.
These identifications are invariant under the action of $\calO_F^{\times, +} / (\calO_{F, {\mathcal{N}}}^\times)^2$, and therefore they induce canonical trivializations:
\begin{equation}
\label{E:trivializations}
\varepsilon_{\tau_{\gothp_i, j}^1} \otimes \varepsilon^{\otimes-p}_{\tau_{\gothp_i, j-1}^{e_i}}  \cong \calO_{\Sh_\FF} \quad \textrm{and} \quad \varepsilon_{\tau_{\gothp_i, j}^l} \otimes \varepsilon^{\otimes-1}_{\tau_{\gothp_i, j}^{l-1}}  \cong \calO_{\Sh_\FF} \textrm{ for } 1<l\leqslant e_i.
\end{equation}
\end{remark}

\begin{construction}
\label{C:m}
For each $\tau = \tau_{\gothp_i,j}^l$ with $l \neq 1$, we define the ``substitute" partial Hasse invariant to be the ``multiplication by $\varpi_i$" map:
\[
\xymatrix@R=0pt{
m_{\varpi_i,j}^{(l)}: \calH_{\gothp_i,j}^{(l)}
\ar[r] &
\calF_{\gothp_i,j}^{(l-1)} / \calF_{\gothp_i,j}^{(l-2)} \cong \dot \omega_{\tau_{\gothp_i, j}^{l-1}}\\
 z\ar@{|->}[r] & [\varpi_i](\tilde z),
}
\]
where $\tilde z$ is a lift to $(\calF_{\gothp_i,j}^{(l-1)})^\perp$ of the local section $z$.
It is straightforward to check that $m_{\varpi_i,j}^{(l)}$ is a well-defined homomorphism. We claim that the map is surjective.
It is enough to check at each (geometric) closed point $\xi$ of $\calM_\FF$. Then $H^1_\dR(\calA_\xi)_{\gothp_i,j} \cong (\overline \FF_p[x]/(x^{e_i}))^{\oplus 2}$, and we may pick a basis so that 
\[
\calF_{\gothp_i,j,\xi}^{(l-1)} \simeq x^a\overline \FF_p[x]/x^{e_i}\overline \FF_p[x] \oplus x^b\overline \FF_p[x]/x^{e_i}\overline \FF_p[x]
\]
with $0<a,b\leq e_i$ and $a+b = 2e_i -l+1$. (Note that this forces $a,b \geq 1$.) Then we have
\[
\calH_{\gothp_i,j,\xi}^{(l)} \simeq x^{a-1}\overline \FF_p[x]/x^{e_i}\overline \FF_p[x] \oplus x^{b-1}\overline \FF_p[x]/x^{e_i}\overline \FF_p[x].
\]
It is clear that $m_{\varpi_i,j}^{(l)}$ takes $\calH_{\gothp_i,j,\xi}^{(l)}$ surjectively onto 
$
\calF_{\gothp_i,j,\xi}^{(l-1)} / \calF_{\gothp_i,j,\xi}^{(l-2)}$.

Restricting $m_{\varpi_i,j}^{(l)}$ to the subbundle $ \calF_{\gothp_i,j}^{(l)} / \calF_{\gothp_i,j}^{(l-1)}\cong \dot \omega_{\tau_{\gothp_i, j}^l}$ induces a section
\[
\dot h_\tau \in H^0(\calM_\FF, \dot \omega_{\tau_{\gothp_i, j}^l}^{\otimes-1} \otimes \dot \omega_{\tau_{\gothp_i, j}^{l-1}}).
\]
We see from the construction that the section $\dot h_\tau$ is invariant under the action of $\calO_F^{\times,+} / (\calO_{F,{\mathcal{N}}}^\times)^2$ and hence
it induces a section
\[
h_\tau \in H^0(\Sh_\FF, \omega^{\otimes-1}_{\tau_{\gothp_i, j}^l} \otimes \omega_{\tau_{\gothp_i, j}^{l-1}}).
\]
By abuse of language, we call $\dot h_\tau$ and $h_\tau$ the \emph{generalized partial Hasse invariants} at $\tau$. We observe that they depend on the choice of uniformizer $\varpi_i$, but the divisors they define do not.
(Notice that all the tensor products above are either over $\Sh_\FF$ or over $\calM_\FF$. We will often omit this information from our notation, whenever it should not create confusion). 
\end{construction}

\begin{remark}\label{R:square-h}
We can view the square $h^2_{\tau}$ as a section of
\begin{equation}
\label{E:square bundle 1}
\omega^{\otimes-2}_{\tau_{\gothp_i, j}^l} \otimes \omega^{\otimes2}_{\tau_{\gothp_i, j}^{l-1}} \otimes \varepsilon_{\tau_{\gothp_i, j}^l} \otimes \varepsilon^{\otimes-1}_{\tau_{\gothp_i, j}^{l-1}}
\end{equation}
via the trivialization \eqref{E:trivializations}. This way, $h^2_\tau$ is a geometric modular form of paritious weight, with normalization
factor $w=0$.
\end{remark}

When $l=1$, the above construction cannot be carried over, and we need a variant 
of the usual construction of the partial Hasse invariants via the Verschiebung maps.

\begin{notation}

The Verschiebung map $V: \calA_\FF^{(p)} \to \calA_\FF$ induces a homomorphism
\[
V_{\gothp_i,j}: \calH^1_\dR(\calA_\FF / \calM_\FF)_{\gothp_i,j} \longto \omega_{\calA_\FF/\calM_\FF,\gothp_i,j-1}^{(p)},
\]
where $\cdot^{(p)}$ denotes the pull-back along the Frobenius map on $\calM_\FF$.
\end{notation}

\begin{construction}
\label{C:Hasse_}
For $\tau = \tau_{\gothp_i,j}^1$, we construct a map
\[
\mathrm{Hasse}_{\varpi_i,j}: \calH_{\gothp_i, j}^{(1)}= \calH^1_\dR(\calA_\FF/\calM_\FF)_{\gothp_i,j}[\varpi_i] \longto
\omega^{(p)}_{\calA_\FF/\calM_\FF,\gothp_i,j-1} /  \big( \calF_{\gothp_i,j-1}^{(e_i- 1)}\big)^{(p)} \cong  \dot \omega_{\tau_{\gothp_i,j -1}^{e_i}}^{\otimes p}
\]
as follows: let $z$ be a local section of $\calH_{\gothp_i,j}^{(1)}$, it belongs to $[\varpi_i]^{e_i-1} \cdot \calH^1_\dR(\calA_\FF/ \calM_\FF)_{ \gothp_i,j}$.
Write $z = [\varpi_i]^{e_i-1} z'$ for a local section $z'$ of $ \calH^1_\dR(\calA_\FF/ \calM_\FF)_{\gothp_i,j}$.
We define $\mathrm{Hasse}_{\varpi_i,j}(z)$ to be the image of $V_{\gothp_i,j}(z')$ in 
$\omega^{(p)}_{\calA_\FF/\calM_\FF,\gothp_i,j-1} /  \big( \calF_{\gothp_i,j-1}^{(e_i- 1)}\big)^{(p)}$.

The ambiguity for the choice of $z'$ lies in the $[\varpi_i]^{e_i-1}$-torsion of $\calH^1_\dR(\calA_\FF/ \calM_\FF)_{\gothp_i,j}$, and hence
any other choice for $z'$ is of the form $z'+[\varpi_i] z''$ for some $z''$ in $\calH^1_\dR(\calA_\FF/ \calM_\FF)_{\gothp_i,j}$.  
Since $\omega^{(p)}_{\calA_\FF/\calM_\FF,\gothp_i,j-1} /  \big( \calF_{\gothp_i,j-1}^{(e_i- 1)}\big)^{(p)}$ is annihilated by $[\varpi_i]$, 
the map $\mathrm{Hasse}_{\varpi_i,j}$ does not depend on such choice.

We now claim that the map $\mathrm{Hasse}_{\varpi_i,j}$ is surjective. Indeed, the Verschiebung map $V_{\gothp_i,j}$ is surjective by definition. Moreover, for any local section $z'$ of $\calH^1_\dR(\calA_\FF/\calM_\FF)_{\gothp_i,j}$, the element $z = [\varpi_i]^{e_i-1}z'$ belongs to $\calH_{\gothp_i,j}^{(1)}$; so the image of $V_{\gothp_i,j}(z')$ in $\omega^{(p)}_{\calA_\FF/\calM_\FF,\gothp_i,j-1} /  \big( \calF_{\gothp_i,j-1}^{(e_i- 1)}\big)^{(p)}$ belongs to the image of $\mathrm{Hasse}_{\varpi_i,j}$. This shows the surjectivity. 

We may restrict the  homomorphism $\mathrm{Hasse}_{\gothp_i,j}$ to the subline bundle $\omega_{\tau_{\gothp_i,j}^1}: = \calF_{\gothp_i,j}^{(1)}$; this induces a section
\[
\dot h_{\tau_{\gothp_i,j}^1} \in H^0(\calM_\FF, \dot \omega_{\tau_{\gothp_i,j}^1}^{\otimes-1} \otimes 
\dot \omega_{\tau_{\gothp_i, j-1}^{e_i}}^{\otimes p}).
\]
By construction, the section $\dot h_{\tau_{\gothp_i,j}^1}$ is invariant under the action of $\calO_F^{\times, +} / (\calO_{F,{\mathcal{N}}}^\times)^2$.
Hence it gives rise to a section
\[
h_{\tau_{\gothp_i,j}^1} \in H^0(\Sh_\FF, \omega^{\otimes-1}_{\tau_{\gothp_i,j}^1} \otimes \omega^{\otimes p}_{\tau_{\gothp_i, j-1}^{e_i}}).
\]
We call $\dot h_{\tau_{\gothp_i,j}^1}$ and $h_{\tau_{\gothp_i,j}^1}$ the \emph{generalized partial Hasse invariants} at $\tau = \tau_{\gothp_i,j}^1$. Once again, these operators depend on the choice of uniformizer $\varpi_i$.
\end{construction}
\begin{remark}\label{R:square-h1}
We can view the square $h^2_{\tau_{\gothp_i,j}^1}$ as a section of
\begin{equation}
\label{E:square bundle 2}
\omega^{\otimes-2}_{\tau_{\gothp_i,j}^1} \otimes \omega^{\otimes 2p}_{\tau_{\gothp_i, j-1}^{e_i}} 
\otimes \varepsilon_{\tau_{\gothp_i,j}^1} \otimes \varepsilon^{\otimes-p}_{\tau_{\gothp_i, j-1}^{e_i}}
\end{equation}
using the trivialization \eqref{E:trivializations}. This way, $h^2_{\tau_{\gothp_i,j}^1}$ is a geometric Hilbert modular form of paritious weight, with normalization factor $w=0$.
\end{remark}


The relation between $\mathrm{Hasse}_{\varpi_i,j}$ and the Verschiebung map is described in the following lemma.  The lemma also explains the reason
for the circumvented definition of the partial Hasse invariant homomorphism $\mathrm{Hasse}_{\varpi_i,j}$, as we wanted to define a ``primitive"
operator.

\begin{lemma}
\label{L:factorization of V}
The Verschiebung map $V_{\gothp_i,j}:\omega_{\calA_\FF/\calM_\FF,\gothp_i,j} \to \omega_{\calA_\FF/\calM_\FF,\gothp_i,j-1}^{(p)}$ induces natural homomorphisms
\[
V_{\gothp_i,j}^{(l)}: \dot \omega_{\tau_{\gothp_i,j}^{l}} =  \calF_{\gothp_i,j}^{(l)} / \calF_{\gothp_i,j}^{(l - 1)} \longto 
\big( \calF_{\gothp_i,j-1}^{(l)} / \calF_{\gothp_i,j-1}^{(l - 1)}\big)^{(p)}
=\dot  \omega_{\tau_{\gothp_i, j-1}^{l}}^{\otimes p}.
\]
Moreover, we have the following commutative diagram:
\[\xymatrix@C=40pt@R=30pt{
\omega_{\calA_\FF/\calM_\FF,\gothp_i,j} / \calF_{\gothp_i,j}^{(e_i - 1)} \ar[r]^-{m^{(e_i)}_{\varpi_i,j}} \ar[d]_{V_{\gothp_i,j}^{(e_i)}} &
\calF_{\gothp_i,j}^{(e_i-1)} / \calF_{\gothp_i,j}^{(e_i - 2)} \ar[r]^-{m^{(e_i-1)}_{\varpi_i,j}} \ar[d]_{V_{\gothp_i,j}^{(e_i-1)}} & \quad \cdots\quad \ar[r]^-{m^{(2)}_{\varpi_i,j}} & \calF_{\gothp_i,j}^{(1)} \ar[d]^{V_{\gothp_i,j}^{(1)}} \ar@{-->}[dlll]^{\quad\quad \mathbf{Hasse}_{\varpi_i,j}}
\\
\big(\omega_{\calA_\FF/\calM_\FF,\gothp_i,j-1} / \calF_{\gothp_i,j-1}^{(e_i - 1)}\big)^{(p)} \ar[r]_-{(m^{(e_i)}_{\varpi_i, j-1})^{(p)}} &\big(
\calF_{\gothp_i,j-1}^{(e_i-1)} / \calF_{\gothp_i,j-1}^{(e_i - 2)}\big)^{(p)} \ar[r]_-{(m^{(e_i-1)}_{\varpi_i, j-1})^{(p)}} &\quad \cdots\quad \ar[r]_-{(m^{(2)}_{\varpi_i, j-1})^{(p)}} &\big( \calF_{\gothp_i,j-1}^{(1)} \big)^{(p)}.
}
\]
In other words,
\[
V_{\gothp_i,j}^{(l)} = (m^{(l+1)}_{\varpi_i,j-1})^{(p)}
\circ \cdots \circ (m^{(e_i)}_{\varpi_i,j-1})^{(p)} \circ \mathrm{Hasse}_{\varpi_i,j} \circ m^{(2)}_{\varpi_i,j}
\circ \cdots \circ m^{(l)}_{\varpi_i,j}.
\]
In terms of functions, the section induced by $V_{\gothp_i,j}^{(l)}$ is equal to
\[
\dot h_{\tau_{\gothp_i,j}^{l}} \cdots \dot h_{\tau_{\gothp_i,j}^{2}} \cdot
\dot h_{\tau_{\gothp_i,j}^{1}} \cdot \dot h_{\tau_{\gothp_i, j-1}^{e_i}}^p \cdots \dot h_{\tau_{\gothp_i, j-1}^{l+1}}^p.
\]
\end{lemma}
\begin{proof}
For the first statement, it suffices to prove that $V_{\gothp_i,j}(\calF_{\gothp_i,j}^{(l)}) \subseteq (\calF_{\gothp_i,j-1}^{(l)})^{(p)}$.
Recall that $ \calF_{\gothp_i,j}^{(l)}$ is annihilated by $[\varpi_i]^l$, and thus it is contained in $[\varpi_i]^{e_i - l}\calH^1_\dR(\calA_\FF/ \calM_\FF)_{\gothp_i, j}$.
For any local section $z$ of $\calF_{\gothp_i,j}^{(l)}$, we can write $z = [\varpi_i]^{e_i-l}y$ for some local section $y$ of $\calH^1_\dR(\calA_\FF/ \calM_\FF)_{\gothp_i, j}$.
Therefore, the image of $V_{\gothp_i,j}(z)$ in $\omega_{\calA_\FF/\calM_\FF,\gothp_i,j-1}^{(p)} /( \calF_{\gothp_i,j-1}^{(l)})^{(p)}$ is of the form $[\varpi_i]^{e_i-l}V_{\gothp_i,j}(y)$.  Since 
$\omega_{\calA_\FF/\calM_\FF,\gothp_i,j-1}^{(p)} /( \calF_{\gothp_i,j-1}^{(l)})^{(p)}$
is annihilated by $[\varpi_i]^{e_i-l}$, we  conclude that
$V_{\gothp_i,j}(z) \in (\calF_{\gothp_i,j-1}^{(l)})^{(p)}$.

To check the commutativity of the above diagram, let $z$ be a local section of $\calF_{\gothp_i,j}^{(l)}$, which can be written as $z = [\varpi_i]^{e_i-l}y$ for some local section $y$ of $\calH^1_\dR(\calA_\FF/\calM_\FF)_{\gothp_i,j}$.
Then $m_{\varpi_i,j}^{(2)} \circ \cdots \circ m_{\varpi_i,j}^{(l)}(z)$ is equal to $[\varpi_i]^{l-1}z \in \calF_{\gothp_i,j}^{(1)}$.
The map $\mathrm{Hasse}_{\varpi_i,j}$ then takes this element to $V_{\gothp_i,j}(y)$.
Thus
\[
(m^{(l+1)}_{\varpi_i,j-1})^{(p)}
\circ \cdots \circ (m^{(e_i)}_{\varpi_i,j-1})^{(p)} \circ \mathrm{Hasse}_{\varpi_i,j} \circ m^{(2)}_{\varpi_i,j}
\circ \cdots \circ m^{(l)}_{\varpi_i,j} (z) = [\varpi_i]^{e_i-l}V_{\gothp_i,j}(y) = V_{\gothp_i,j}^{(l)}(z).\qedhere
\]
\end{proof}
\subsection{The stratification induced by the $h_\tau$}
We now investigate some properties of the stratification induced on the special fiber of the splitting model by the generalized partial Hasse invariants.

For each $p$-adic embedding $\tau \in \Sigma$, denote by $Z_\tau$ the zero locus of $h_\tau$ on $\Sh_\FF$.
In general, for a subset $\calT\subseteq\Sigma$, we set $Z_\calT := \cap _{\tau \in \calT} Z_\tau$, with the convention
that if $\calT$ is the empty set, this intersection is interpreted to be the entire space $\Sh_\FF$.
We define $\dot Z_\tau$ and $\dot Z_{\calT}$ on $\calM_\FF$ similarly.
We remark that, although the generalized partial Hasse-invariants $h_\tau$ and $\dot h_\tau$ depend on the choice of uniformizers $\varpi_i$, their zero loci $Z_\tau$ and $\dot Z_\tau$ do not.
We have the following result:

\begin{theorem}
\label{T:smoothness of ramified GO strata}
The closed subschemes $Z_\tau$ (resp. $\dot Z_\tau$) are proper and smooth divisors with simple normal crossings on $\Sh_\FF$ (resp. $\calM_\FF$).
Moreover, the sheaf of relative differentials $
\Omega^1_{Z_\calT / \FF}$ (resp. $
\Omega^1_{\dot Z_\calT / \FF}$) admits a canonical filtration whose subquotients are exactly given by the sheaves $\omega_\tau^{\otimes 2} \otimes \varepsilon_\tau^{-1}$ (resp. $\dot \omega^{\otimes 2}_\tau \otimes\dot  \varepsilon^{-1}_\tau$)  for $ \tau \in \Sigma - \calT$.
\end{theorem}
\begin{proof}
We first prove the properness of $\dot Z_\tau$. The Verschiebung morphism on the semi-abelian variety $\calA^\tor_\FF \to \calM^\tor_\FF$ induces a map
\begin{equation}
\label{E:map V}
V: \wedge^g_{\calO_{\calM^\tor_\FF}} \dot\omega_{\calA^\tor_\FF / \calM^\tor_\FF} \to \wedge^g_{\calO_{\calM^\tor_\FF}} \dot\omega_{\calA^\tor_\FF / \calM^\tor_\FF}^{(p)},
\end{equation}
and hence defines a section $\dot h_\mathrm{tot}$ (the total Hasse invariant) of $\dot \omega^{(\textbf{p-1}, p-1)}$ on $\calM^\tor_\FF$. By Lemma~\ref{L:factorization of V},  the restriction of $\dot h_\mathrm{tot}$ to $\calM_\FF$ is a product of appropriate \emph{positive} powers of each of the generalized partial Hasse invariants $\dot h_\tau$ for $\tau\in\Sigma$.
Since $\calA_\FF^\tor$ is a torus over the boundary divisors $\dot {\mathtt D}$ (cf. \cite[Proposition~7.6]{dimitrov}), the Verschiebung is an isomorphism over $\dot {\ttD}$. So $\dot h_\mathrm{tot}$ and hence each $\dot h_\tau$ is non-vanishing over $\dot \ttD$. The properness of $\dot Z_\tau$ and of $Z_\tau$ follows.

To complete the proof of the theorem, we argue via Grothendieck-Messing deformation theory (cf. the proof of Theorem~\ref{T:Kodaira-Spencer ramified}).
Let $S_0 \hookrightarrow S$ be a closed immersion of locally Noetherian $\FF$-schemes whose ideal of definition $\calI$ satisfies $\calI^2 = 0$.
Let $x_0=(A_0, \lambda_0, i_0, \underline \scrF)$ be an $S_0$-point of $\dot Z_{\calT}$.  Lifting this point to an $S$-point of $\calM_{ \FF}$ is
equivalent to lift, for each $\tau_{\gothp_i,j}$, the filtration
\[
0 =\scrF_{\gothp_i, j}^{(0)} \subsetneq \scrF_{\gothp_i, j}^{(1)} \subsetneq \cdots \subsetneq
\scrF_{\gothp_i, j}^{(e_i)} = \omega_{A_0/S_0, \gothp_i, j} \subset \calH^1_\cris(A_0 /S_0)_{S_0,\gothp_i,j}
\]
to a $[\varpi_i]$-stable filtration
\[
0 =\tilde \scrF_{\gothp_i, j}^{(0)} \subsetneq \tilde \scrF_{\gothp_i, j}^{(1)} \subsetneq \cdots \subsetneq
\tilde \scrF_{\gothp_i, j}^{(e_i)} = \tilde \omega_{\gothp_i, j} \subset \calH^1_\cris(A_0 /S_0)_{S, \gothp_i,j}
\]
such that
each subquotient $\tilde \scrF_{\gothp_i, j}^{(l)} / \tilde \scrF_{\gothp_i, j}^{(l-1)}$ is a locally free $\calO_S$-module of rank annihilated by $[\varpi_i]$.

We claim that a given $S$-lift $x=(A,\lambda,i,\underline {\tilde \scrF})$ of $x_0$ endowed with filtration $\underline{\tilde \scrF}=(\tilde \scrF_{\gothp_i,j}^{(l)})_{i,j,l}$ lies in $\dot Z_\calT$ if and only if for each $\tau^l_{\gothp_i,j}\in\calT$ the sheaf $\tilde \scrF_{\gothp_i,j}^{(l)}$ equals some \emph{fixed} lift of $\scrF_{\gothp_i,j}^{(l)}$. 
This claim, together with (the proof of) Theorem \ref{T:Kodaira-Spencer ramified}, concludes the proof of this Theorem; in particular, it explains the ``missing" $\dot \omega_\tau^{\otimes 2} \otimes \dot \varepsilon^{\otimes-1}_\tau$'s in the subquotients of the Kodaira-Spencer filtration of the sheaf of differentials $\Omega^1_{\dot Z_\calT/\FF}$.

To prove the claim, we proceed inductively on $l$ and considering each $\tau = \tau_{\gothp_i,j}^l$ independently.
When $\tau \notin \calT$ there is nothing to check, so we assume that $\tau = \tau_{\gothp_i,j}^l\in \calT$, and we distinguish 
two cases: $l = 1$ or $l > 1$.

When $l= 1$, the lift $x$ of $x_0$ belongs to $\dot Z_\tau$ if and only if $\tilde \scrF_{\gothp_i,j}^{(1)}$ is contained in the kernel of
\[
\mathrm{Hasse}_{\varpi_i,j}: \calH^1_\cris(A_0/S_0)_{S, \gothp_i,j} [\varpi_i] \longto \big(\omega_{A/S, \gothp_i,j-1} / \tilde\scrF_{\gothp_i,j-1}^{(e_i-1)} \big)^{(p)},
\]
where $\cdot[\varpi_i]$ denotes $\varpi_i$-torsion.  
Since the above map is surjective by Construction~\ref{C:Hasse_}, its kernel is a rank one $\calO_S$-subbundle which coincides with $\tilde \scrF_{\gothp_i,j}^{(1)}$ if and only if
$x\in \dot Z_\tau$. This shows that the $\tilde \scrF_{\gothp_i,j}^{(1)}$ is uniquely determined by the condition $x\in \dot Z_\tau$.

When $l >1$, the lift $x$ of $x_0$ belongs to $\dot Z_\tau$ if and only if $\tilde \scrF_{\gothp_i,j}^{(l)}$ is contained in the kernel of the surjective map
\[
m_{\varpi_i,j}^{(l)}: \scrH_{\gothp_i,j}^{(l)} \longto \tilde \scrF_{\gothp_i,j}^{(l-1)}
/\tilde \scrF_{\gothp_i,j}^{(l-2)}.
\]
So the kernel of
$m_{\varpi_i,j}^{(l)}$ on $\scrH_{\gothp_i,j}^{(l)}$ is a rank one $\calO_S$-subbundle, and that it coincide with  $\tilde \scrF_{\gothp_i,j}^{(l)}$ if
and only if $x\in \dot Z_\tau$. This shows that the $\tilde \scrF_{\gothp_i,j}^{(l)}$ is uniquely determined by the condition $x\in \dot Z_\tau$.

The result for $Z_\calT$ follows by passing to the quotient for the free action of $\calO_F^{\times, +} / (\calO_{F, {\mathcal{N}}}^\times)^2$.
\end{proof}

\begin{remark}
The stratification given by the $\dot Z_\tau$'s should correspond to a certain cell-decomposition of the twisted product of projective spaces $\Gr_1 \tilde \times \cdots \tilde \times \Gr_1$ in Remark~\ref{R:dependence-on-embeddings}(2).

It would be interesting to compare the stratification induced by the $\dot Z_\tau$'s and their intersections with the stratification considered by Sasaki in \cite{sasaki}.
\end{remark}

We now construct a suitable ``modular form" $b_\tau$ defined and nowhere vanishing on the zero set $Z_\tau$ of $h_\tau$. The forms $b_\tau$'s extend the operators constructed in \cite[Section 3.2]{emerton-reduzzi-xiao} under the assumption that $p$ was unramified in $F$.

For a $p$-adic embedding $\tau  = \tau_{\gothp_i,j}^l$,
denote by $\calI_\tau$ the ideal sheaf associated to the closed embedding $Z_\tau \hookrightarrow \Sh_\FF$.
By Theorem~\ref{T:smoothness of ramified GO strata}, the section $h_\tau$ vanishes with simple zeros along $Z_\tau$, so that it defines an invertible function
\begin{align*}
b_\tau &\in H^0(Z_\tau, \omega^{\otimes-1}_{\tau_{\gothp_i,j}^1}
\otimes \omega^{\otimes p}_{\tau_{\gothp_i,j-1}^{e_i}}
\otimes \calI_\tau / \calI^2_\tau) ,\quad \textrm{if } l =1;\\
b_\tau &\in H^0(Z_\tau, \omega^{\otimes-1}_{\tau_{\gothp_i, j}^{l}}
\otimes \omega_{\tau_{\gothp_i,j}^{l-1}}
\otimes \calI_\tau / \calI^2_\tau), \quad \textrm{if } l >1.
\end{align*}
Using the canonical exact sequence \[
0 \to \calI_\tau / \calI^2_\tau
\to \Omega^1_{\Sh_\FF/\FF}|_{Z_\tau} \to \Omega^1_{Z_\tau /\FF} \to 0
\]
together with Theorem~\ref{T:Kodaira-Spencer ramified} and the proof of Theorem \ref{T:smoothness of ramified GO strata}, we see that $\calI_\tau/ \calI^2_\tau \cong \omega^2_\tau \otimes \varepsilon^{\otimes-1}_\tau$. In particular, we can see the above \emph{generalized partial $b$-functions} as \emph{nowhere vanishing} sections:
\[
b_\tau \in H^0(Z_\tau, \omega_{\tau_{\gothp_i,j}^1}
\otimes \omega^{\otimes p}_{\tau_{\gothp_i,j-1}^{e_i}}  \otimes \varepsilon^{\otimes-1}_{\tau_{\gothp_i,j}^1}) , \textrm{ if } l =1;\quad\quad
b_\tau \in H^0(Z_\tau, \omega_{\tau_{\gothp_i, j}^{l}}
\otimes \omega_{\tau_{\gothp_i,j}^{l-1}}\otimes \varepsilon^{\otimes-1}_{\tau_{\gothp_i,j}^l}), \textrm{ if } l >1.
\]
\begin{notation}
Let $\{\bfe_\tau\}_{\tau \in \Sigma}$ denote the standard $\ZZ$-basis of $\ZZ^\Sigma$ and set, for each $\tau = \tau_{\gothp_i, j}^l$:
\begin{eqnarray*}
\quad \bfp_{\tau_{\gothp_i,j}^l} : = -\bfe_{\tau_{\gothp_i,j}^l} + \bfe_{\tau_{\gothp_i,j}^{l-1}}&, \quad & \bfq_{\tau_{\gothp_i,j}^l} : = \bfe_{\tau_{\gothp_i,j}^l} + \bfe_{\tau_{\gothp_i,j}^{l-1}}, \textrm{    when }l > 1, 
\\
\quad \bfp_{\tau_{\gothp_i,j}^1} : = -\bfe_{\tau_{\gothp_i,j}^1} +p \bfe_{\tau_{\gothp_i,j-1}^{e_i}}&, \quad & \bfq_{\tau_{\gothp_i,j}^1} : = \bfe_{\tau_{\gothp_i,j}^1} + p\bfe_{\tau_{\gothp_i,j-1}^{e_i}}, \textrm{ when }l = 1.
\end{eqnarray*}
If $\bfk = \sum_{\tau \in\Sigma} k_\tau \bfe_\tau \in \ZZ^\Sigma$ and $w\in \ZZ$ are such that $(\bfk,w)$ is a paritious weight, we set:
\[
\omega^{(\bfk, w)} = \bigotimes_{\tau \in \Sigma} \omega^{k_\tau}_\tau \otimes \varepsilon^{(w-k_\tau)/2}_\tau.
\]
For example, the \emph{square} $h^2_\tau$ of the generalized partial Hasse invariant $h_\tau$ can be viewed as a section of $\omega^{(2\bfp_\tau,0)}_\FF$ (cf. Remarks \ref{R:square-h} and \ref{R:square-h1}).
Similarly, the square $b^2_\tau$ can be viewed as a section of $\omega^{(2\bfq_\tau,0)}_\FF$ over $Z_\tau$.
\end{notation}

\subsection{Construction of liftings}
Suitable powers of $h_\tau$ and $b_\tau$ can be lifted to non-reduced neighborhoods of the strata $Z_\calT$. We recall this below: the reader might consult \cite[Sections 3.3.1--3.3.6]{emerton-reduzzi-xiao} for details. We remind the reader of our convention of dropping the superscript $\PR$ from our notation.

Fix a positive integer $m \geq 1$ and set $R_m: = \calO / (\varpi^m).$
\subsubsection{Liftings of $h_\tau^2$}
\label{S:lifting of h}
For any $\tau \in \Sigma$ and any positive integer $M$ divisible by $2p^{m-1}$,
there exists a unique element
\[
\tilde h_{\tau, M} \in H^0(\Sh^\tor_{R_m}, \omega^{(M \bfp_\tau, 0)}_{R_m})
\]
which is locally the $\frac M2$th power of a lift of $h^2_\tau$ as a section of $\omega^{(2\bfp_\tau,0)}_\FF$.
Clearly, we have $\tilde h_{\tau, M_1} \tilde h_{\tau, M_2} = \tilde h_{\tau, M_1 + M_2}$ for positive integers  $M_1$ and $M_2$ divisible by $2p^{m-1}$.

In general, if $\bfM = (M_\tau)_{\tau \in \Sigma} \in \ZZ^\Sigma$ is a tuple of non-negative integers all divisible by $2p^{m-1}$, we define
\[
\tilde h_\bfM: = \prod_{\tau \in \Sigma} \tilde h_{\tau, M_\tau}, \quad \textrm{and} \quad \bfM_\bfp: = \sum_{\tau \in \Sigma} M_\tau \bfp_\tau,
\]
with the convention that $\tilde h_{\tau, 0} = 1$.  Then $\tilde h_{\bfM} $ is a Hilbert modular form of paritious weight $(\bfM_\bfp, 0)$.

\begin{lemma}
\label{L:lift of partial Hasse invariant}
Let $\bfM =(M_\tau)_{\tau \in \Sigma} \in (2p^{m-1} \ZZ_{\geq0})^\Sigma$ be a tuple of non-negative integers all divisible by $2p^{m-1}$.  For any paritious weight $(\bfk, w) \in \ZZ^\Sigma \times \ZZ$, multiplication by $\tilde h_\bfM$ induces a Hecke-equivariant\footnote{Here being Hecke-equivariant is in the sense of Section~\ref{S:tame hecke operator}.} morphism of sheaves:
\[
\cdot \tilde h_\bfM:
\omega^{(\bfk, w)}_{R_m} \longrightarrow
\omega^{(\bfk  + \bfM_\bfp, w)}_{R_m}.
\]
\end{lemma}
\begin{proof}
This follows from the same argument as in \cite[Lemma~3.3.2]{emerton-reduzzi-xiao}, using the description \eqref{E:tame hecke action} of Hecke action on test objects.
\end{proof}

For a positive integer $M_\tau$ divisible by $2p^{m-1}$, we denote by $Z_{M_\tau \bfe_\tau}$ the closed subscheme of $\Sh^\tor_{R_m}$ defined by the vanishing of $\tilde h_{M_\tau\bfe_\tau} = \tilde h_{\tau, M_\tau}$; we set $Z_{0 \bfe_\tau}: = \Sh^\tor_{R_m}$.
For general $\bfM =(M_\tau)_{\tau \in \Sigma} \in (2p^{m-1} \ZZ_{\geq0})^\Sigma$,
we set
\[
Z_\bfM: = \bigcap_{\tau \in \Sigma} Z_{M_\tau \bfe_\tau}.
\]
We say the \emph{support} of $\bfM$ is the subset $|\bfM|: = \{\tau \in \Sigma\,|\, M_\tau =0\}$
of $\Sigma$.  The \emph{dimension} of $\bfM$ is defined to be $\dim(\bfM): = \#|\bfM|$; it is the Krull dimension of $Z_{\bfM}$ by Theorem~\ref{T:smoothness of ramified GO strata}.

\subsubsection{Liftings of $b_\tau^2$}
Fix a place $\tau \in \Sigma$ and two positive integers $M$ and $T$ divisible by $2p^{m-1}$ such that $T > M + 2p^{m-1}$.
Then there exists a unique element
\[
\tilde b_{\tau, M, T} \in H^0(Z_{M\bfe_\tau}, \omega^{(T\bfq_\tau,0)}_{R_m})
\]
which is locally the $\frac T2$th power of a lift of $b^2_\tau$ as a section of $\omega^{(2\bfq_\tau,0)}_\FF$.
The elements $\tilde b_{\tau, M, T}$ satisfy the obvious compatibility conditions when changing $M$ and $T$.

In general, let $\bfM = (M_\tau)_{\tau \in \Sigma}, \bfT = (T_\tau)_{\tau \in \Sigma} \in \ZZ^\Sigma_{\geq 0}$ be two tuples of non-negative integers all divisible by $2p^{m-1}$.
Assume that if $M_\tau=0$ then $T_\tau=0$, and that if $M_\tau >0$ then either $T_\tau =0$ or $T_\tau > M_\tau + 2p^{m-1}$.
We set
\[
\tilde b_{\bfM, \bfT}: = \prod_{\tau \in \Sigma} \tilde b_{\tau, M_\tau, T_\tau},
\]
with the convention that $\tilde b_{\tau, M_\tau, T_\tau} = 1$ if $M_\tau$ or $T_\tau$ is zero.
When no ambiguity arises, we write $\tilde b_{\bfT}: = \tilde b_{\bfM, \bfT}$.
In particular, when $\bfM = 0$, our conventions imply that $\tilde b_{\bfM, \bfT}$ is the identity function.

\begin{lemma}
Let $\bfM, \bfT \in (2p^{m-1}\ZZ_{\geq 0})^\Sigma$ be such that if $M_\tau=0$ then $T_\tau=0$, and if $M_\tau >0$ then either $T_\tau =0$ or $T_\tau > M_\tau + 2p^{m-1}$.
For any paritious weight $(\bfk, w) \in \ZZ^\Sigma \times \ZZ$ there is a Hecke equivariant isomorphism of sheaves on $Z_\bfM$:
\[
\cdot \tilde b_{\bfT}: \ \omega^{(\bfk, w)}_{R_m}(-\ttD)_{|Z_\bfM} \xrightarrow{\ \cong \ }
\omega^{(\bfk + \sum_\tau T_\tau \bfq_\tau, w)}_{R_m}(-\ttD)_{|Z_\bfM}
\]

\end{lemma}

\begin{proof}
Cf. \cite[3.3.6]{emerton-reduzzi-xiao}.
\end{proof}

\section{Application to Galois representations}
Extending the strategy of \cite{emerton-reduzzi-xiao}, we show that the existence of the generalized partial Hasse invariants and of the
trivializations $b_\tau$ over the splitting models allows one to attach pseudo-representations to Hecke eigenclasses occurring in the torsion cohomology of $\Sh^{\PR,\tor}_{R_m}$.
\subsection{Universal ring for pseudo-representations}
We recall the construction of universal ring of pseudo-representations from \cite[Section~4.1]{emerton-reduzzi-xiao}.

Let $\ttS$ denote a finite set of places of $F$ containing the places dividing $p{\mathcal{N}}$ the the archimedean places.
Let $G_{F, \ttS}$ denote the Galois group of a maximal algebraic extension of $F$ that is unramified outside $\ttS$.
Write $G_{F, \ttS}$ as an inverse limit $G_{F, \ttS} = \varprojlim_i G_i$ of finite groups.
The \emph{universal ring for two-dimensional continuous pseudo-representation} of $G_{F, \ttS}$ with values in an $\calO$-algebra is the
inverse limit
$\calR_{G_{F, \ttS}}^\ps: = \varprojlim_i \calR_{G_i}^\ps$,
where each $\calR_{G_i}^\ps$ is the
 quotient of the polynomial ring $\calO[t_g: g \in G_i]$ by the ideal generated by
\begin{align*}
&  t_{1}-2,\ t_{g_{1}g_{2}}-t_{g_{2}g_{1}}\text{ for }g_{1},g_{2}\in G_i\text{,
and }\\
&  t_{g_{1}}t_{g_{2}}t_{g_{3}}+t_{g_{1}g_{2}g_{3}}+t_{g_{1}g_{3}g_{2}%
}-t_{g_{1}}t_{g_{2}g_{3}}-t_{g_{2}}t_{g_{1}g_{3}}-t_{g_{3}}t_{g_{1}g_{2}%
}\text{ for }g_{1},g_{2},g_{3}\in G_i.
\end{align*}

Let $\TT_\ttS^\univ = \calO[t_\mathfrak{q}: \mathfrak{q} \notin \ttS]$ denote the universal Hecke algebra (cf. Notation \ref{N:universal hecke}).  There is a natural homomorphism of $\calO$-algebras with dense image:
\[
\mathbb{T}_{\text{\texttt{S}}}^{\operatorname*{univ}}\rightarrow
\mathcal{R}_{G_{F,\text{\texttt{S}}}}^{\mathrm{ps}},\quad t_{\mathfrak{q}}\mapsto
t_{\mathrm{Frob}_{\mathfrak{q}}},\text{ for }\mathfrak{q}\notin\text{\texttt{S}}.
\]

A $\TT_\ttS^\univ$-module $M$ is said to be of \emph{Galois type} if the action of $\TT_\ttS^\univ$
factors through the
image of $\mathbb{T}_{\text{\texttt{S}}}^{\operatorname*{univ}}\rightarrow
\mathcal{R}_{G_{F,\text{\texttt{S}}}}^{\mathrm{ps}}$ and extends by continuity
to an action of $\mathcal{R}_{G_{F,\text{\texttt{S}}}}^{\mathrm{ps}}$.
In general, a bounded complex in the category of $\TT_\ttS^\univ$-modules is said to be of Galois type if each of its terms is of Galois type.

Clearly, the kernel and cokernel of continuous morphisms of $\TT_\ttS^\univ$-modules of Galois type are of Galois type; in particular, all
cohomology groups of a complex of $\mathbb{T}_{\text{\texttt{S}}%
}^{\operatorname*{univ}}$-modules of Galois type are of Galois type.
\subsection{Existence of favorable resolutions}
Recall that we have fixed a positive integer $m$ and that we have set $R_m=\calO/(\varpi^m)$.
For each $\bfM \in (2p^{m-1}\ZZ_{\geq 0})^\Sigma$, the lift $\tilde h_{\bfM}$ of product of generalized partial Hasse invariants induces a Hecke equivariant map between automorphic sheaves (Lemma~\ref{L:lift of partial Hasse invariant}), and its zero set $Z_\bfM$ is stable under the action of the tame Hecke operators.

For a tuple $\bfM \in (2p^{m-1}\ZZ_{\geq 0})^\Sigma$ and a paritious weight $(\bfk, w) \in \ZZ^\Sigma \times \ZZ$, we say that $(\bfk,w)$ is a \emph{favorable weight} with respect to $\bfM$ if:
\begin{itemize}
\item $H^0(Z_\bfM, \omega^{(\bfk, w)}_{R_m}(-\ttD))$ is of Galois type, and
\item $H^i(Z_\bfM, \omega^{(\bfk, w)}_{R_m}(-\ttD)) = 0$ for all $i>0$.
\end{itemize}
In this case, we also say that $\omega^{(\bfk,w)}_{R_m}(-\ttD)|_{Z_\bfM}$ is a \emph{favorable sheaf}.

\begin{notation}
For $\ttN \in \ZZ$, set $\underline \ttN = \ttN \cdot \sum_{\tau \in \Sigma} \bfe_\tau$.
We set:
\[
\underline \ex = \sum_{i=1}^r \sum_{j=1}^{f_i} \sum_{l = 1}^{e_i} 2(l-e_i-1) \bfe_{\tau_{\gothp_i,j}^l}.
\]
\end{notation}

\begin{lemma}
\label{L:numerics}
There exists $\ttN_0 \in \ZZ_{>0}$ such that for any $\ttN \geq \ttN_0$ and any subset $J \subseteq \Sigma$, the point $\underline \ttN + \mathrm{\underline{ex}}$ lies in the interior of the positive cone $\mathrm{Cone}_J$ spanned in $\RR^\Sigma$ by the set $\{\bfp_\tau, \bfq_{\tau'}; \tau \in J, \tau' \notin J\}$.
\end{lemma}
\begin{proof}
It is clear from the definition of $\bfp_\tau$ and $\bfq_{\tau'}$ that the line $\RR_{\geq 1} \underline 1$ lies in the interior of the positive cone $\mathrm{Cone}_J$ for each $J$.
There are only finitely many possible $J$'s, so this line lies in the interior of the intersection of the corresponding cones.
The lemma follows.
\end{proof}
We suppose from now on that we have fixed an integer $\ttN_0$ as in Lemma \ref{L:numerics}.
\begin{lemma}
\label{L:ampleness}
Let $\calM^{\PR, *}$ denote the $\calO$-scheme obtained by gluing $\calM^\PR$ with the minimal compactification $\calM^{\DP,*}$ of $\calM^\DP$ over the Rapoport locus $\calM^\Ra$.
Then there exists an even integer $\ttN \geq \ttN_0$ such that the line bundle $\dot\omega^{(\underline \ttN + \underline{\mathrm{ex}},0)}_\FF$ on $\calM^{\PR,\tor}_\FF$ descends to an ample line bundle on $\calM^{\PR, *}_\FF$.
Similarly, for the same $\ttN$, $\omega^{(\underline \ttN + \underline{\mathrm{ex}}, 0)}_\FF$ descends to an ample line bundle on $\Sh^{\PR, *}_\FF$, the quotient of $\calM^{\PR, *}_\FF$ by the action of $\calO_F^{\times, +} / (\calO_{F, {\mathcal{N}}}^\times)^2$.
\end{lemma}
\begin{proof}
The sheaf $\dot\omega^{(\mathbf{2},0)}_{\FF}$ over $\calM^{\PR,\tor}_\FF$ descends to an invertible sheaf on $\calM^{\DP,\tor}_\FF$ (because it has parallel weight, cf. Notation \ref{N:M Sh A}), and then further 
to an \emph{ample} invertible sheaf $\dot\omega^{(\mathbf{2},0)}_{\FF,\mathrm{min}}$ on $\calM^{\DP,*}_\FF$. Moreover, the sheaf $\dot\omega^{(\underline\ex,0)}_\FF$ descends to a line bundle $\dot\omega^{(\underline\ex,0)}_{\FF,\mathrm{min}}$ over
$\calM^{\PR,*}_\FF$ which is relatively ample with respect to the natural map $\pi:\calM^{\PR,*}_\FF \to \calM^{\DP,*}_\FF$ (cf. Lemma \ref{L:relative ample}). It follows from \cite[Proposition 1.7.10]{lazarsfeld} that there exists a positive even integer 
$\ttN \geq \ttN_0$ such that $\dot\omega^{(\underline{\mathrm{ex}},0)}_{\FF,\mathrm{min}} \otimes 
\pi^*(\dot\omega^{(\mathbf2,0)}_{\FF,\mathrm{min}})^{\otimes (\ttN/2)}$ is ample on $\calM^{\PR,*}_\FF$. 

The result passes to the quotient by the action of $\calO_F^{\times, +} / (\calO_{F, {\mathcal{N}}}^\times)^2$ because ampleness can be detected after a finite surjective map \cite[Corollary 1.2.28]{lazarsfeld}.
\end{proof}

We suppose from now on that we have fixed an even integer $\ttN$ as in Lemma~\ref{L:ampleness}.

\begin{proposition}
\label{P:ampleness}
For any paritious weight $(\bfk, w) \in\ZZ^\Sigma \times \ZZ$, there is an integer $n_0 = n_0(\bfk,w)$ such that for any $n \geq n_0$ and any $i>0$, we have
\[
H^i\big(\Sh^{\PR, \tor}, \omega^{(\bfk + n\cdot (\underline \ttN + \underline {ex}),w)}(-\ttD) \big) =0
\]
\end{proposition}
\begin{proof}
This follows from the exact same argument as in \cite[Lemma~4.2.2]{emerton-reduzzi-xiao}, using the ampleness from Lemma~\ref{L:ampleness}.
\end{proof}

\begin{proposition}
\label{P:existence of change weights}
Let $\bfM \in (2p^{m-1}\ZZ_{\geq 0})^\Sigma$ be a tuple having support $J \subseteq \Sigma$, and let $(\bfk_1, w), \dots, (\bfk_t, w)\in \ZZ^\Sigma \times \ZZ$ be paritious weights (with the same normalization factor $w$).
Then there exists a tuple
\[
\sum_{\tau \in J} N_\tau \bfe_\tau \in (2p^{m-1}\ZZ_{> 0})^J
\]
such that the paritious weight $(\bfk_\alpha + \sum_{\tau \in J} N_\tau \bfp_\tau,w)$ is favorable with respect to $\bfM$ for every $\alpha =1,\dots, t$.
\end{proposition}
\begin{proof}
This can be proved in the same way as \cite[Lemma~4.3.2 and Lemma~4.3.3]{emerton-reduzzi-xiao}, using the combinatorial input from Lemma~\ref{L:numerics} and the vanishing result of Proposition~\ref{P:ampleness}.
\end{proof}

Given two tuples $\mathbf{M},\mathbf{M}^{\prime}\in\left(2p^{m-1}%
\mathbb{Z}_{\geq0}\right)  ^{\Sigma}$ and two paritious weights $(\bfk, w),(\bfk', w)
\in\mathbb{Z}^{\Sigma} \times \ZZ$ with the same $w$, a
homomorphism of sheaves of $\mathcal{O}_{\Sh_{R_m}^{\PR,\operatorname*{tor}}}%
$-modules
\[
\xi:\omega_{R_m}^{(\bfk, w)}(-\mathtt{D})_{|Z_{\mathbf{M}}}\rightarrow
\omega_{R_m}^{(\bfk',w)}(-\mathtt{D})_{|Z_{\mathbf{M}^{\prime}}}%
\]
is called \emph{admissible} (cf. \cite[Definition~4.4.1]{emerton-reduzzi-xiao}) either if it is the zero homomorphism, or if the
following three conditions are satisfied:

\begin{itemize}
\item $\mathbf{k}^{\prime}-\mathbf{k}=\mathbf{N}_{\mathbf{p}}$ for some
$\mathbf{N}\in\left(  2p^{m-1}\mathbb{Z}_{\geq0}\right)  ^{\Sigma}$ (cf.\ \ref{S:lifting of h}
for the meaning of this notation),

\item $\xi$ is induced by multiplication by $\alpha\tilde{h}_{\mathbf{N}}$ for
some $\alpha\in R_{m}$, and

\item for each $\tau$ such that $M_\tau>0$, we have $M_\tau^{\prime}>0$ and
$M_\tau+(k_\tau^{\prime}-k_\tau)\geq M_\tau^{\prime}$.
\end{itemize}
Any such admissible homomorphism is $\TT_\ttS^\univ$-equivariant. An admissible complex is a bounded complex of sheaves of $\calO_{\Sh^{\PR,\tor}_{R_m}}$-modules in which each term is a finite direct sum of automorphic sheaves of paritious weight of the form $\omega_{R_m}^{(\bfk, w)}(-\mathtt{D})_{|Z_{\mathbf{M}}}$ with $w$ fixed, and whose differentials are given by matrices of admissible homomorphisms.

\begin{theorem}
\label{T:main}
Let $(\bfk, w) \in \ZZ^\Sigma \times \ZZ$ be a paritious weight.  There exists an admissible complex $C^{\bullet}$ quasi-isomorphic to $\omega^{(\bfk,w)}_{R_m}(-\mathtt{D})$ such that all terms of $C^{\bullet}$ are
favorable. Hence there exists a bounded complex of $\mathbb{T}^\univ
_{\ttS}$-modules of Galois type whose cohomology groups are
$H^{\bullet}(\Sh^{\PR,\operatorname*{tor}}_{R_m},\omega^{(\bfk,w)}_{R_m}
(-\mathtt{D}))$.
\end{theorem}
\begin{proof}
This can be proved in the same way as \cite[Theorem~4.4.4]{emerton-reduzzi-xiao}.
Indeed, the proof in \emph{loc.cit.} relies on \cite[Lemma~4.4.5]{emerton-reduzzi-xiao}, which makes use of the following two facts:
\begin{itemize}
\item
the divisors $Z_\tau$ have simple normal crossings, which is proved in the context of splitting models in Theorem~\ref{T:smoothness of ramified GO strata}, and
\item
the existence of suitable tuples of integers $\bfN^{(J)}$ (cf. Lemma 4.4.5 in \emph{loc.cit.}) which allow one to construct the favorable resolution; such tuples are constructed in our context in Proposition~\ref{P:existence of change weights}.\qedhere
\end{itemize}
\end{proof}

\begin{corollary}
\label{C:main}
Let $(\mathbf{k}, w)\in\mathbb{Z}^\Sigma \times \ZZ$ be a paritious weight. Any
cuspidal Hecke eigenclass $c\in H^{\bullet}(\Sh^{\PR,\operatorname*{tor}}_{R_m}
,\omega_{R_m}^{(\mathbf{k}, w)}(-\mathtt{D}))$ has canonically attached a continuous,
$R_{m}$-linear, two-dimensional pseudo-representation $\tau_{c}$ of the Galois
group $G_{F,\ttS}$ such that
\[
\tau_{c}(\operatorname*{Frob}\nolimits_{\gothq}) =a_{\gothq}%
\]
for all finite primes $\gothq$ of $F$ outside $\ttS$, where $T_{\gothq
}c=a_{\gothq}c$.
\end{corollary}

\end{document}